\documentclass[11pt]{amsproc}

\usepackage[hmargin=3.2cm,vmargin=3.4cm]{geometry}
\usepackage{amsmath,amssymb,amsthm,mathtools}
\usepackage{enumitem, todonotes}
\usepackage{tikz}
\usepackage{amsrefs}
\usepackage[colorlinks=true,linkcolor=blue,urlcolor=black]{hyperref}
\usepackage{todonotes}

\newcommand{\BH}{B(H)}

\newcommand{\QH}{Q(H)}

\newcommand{\bbN}{\mathbb N}

\newcommand{\bbT}{\mathbb T}

\DeclareMathOperator{\dist}{dist}
\DeclareMathOperator{\SU}{SU}
\DeclareMathOperator{\WOTlim}{WOT-\lim}
\DeclareMathOperator{\dimnuc}{dim_{nuc}}
\DeclareMathOperator{\Eval}{Eval}

\DeclareMathOperator{\Ad}{Ad}
\DeclareMathOperator{\Sep}{Sep}
\newcommand{\cstar}{C*}

\newcommand{\cst}{\mathsf{C^*}}

\newcommand{\sfA}{\mathsf A}
\newcommand{\sfAR}{\mathsf A\mathcal R}
\newcommand{\sfB}{\mathsf B}
\newcommand{\sfC}{\mathsf C}

\newcommand{\sfD}{\mathsf D}

\newcommand{\sfH}{\mathsf {H}}
\DeclareMathOperator{\sfAC}{\mathsf {AC}}
\DeclareMathOperator{\sfvN}{\mathsf {vN}}
\DeclareMathOperator{\sfAF}{\mathsf {AF}}
\DeclareMathOperator{\sfFD}{\mathsf {FD}}
\DeclareMathOperator{\sfUC}{\mathsf {uAb}}
\DeclareMathOperator{\sfCOM}{\mathsf {Ab}}

\newcommand{\cC}{\mathcal C}
\newcommand{\cB}{\mathcal B}

\newcommand{\cR}{\mathcal R}
\newcommand{\cM}{\mathcal M}
\newcommand{\cQ}{\mathcal Q}
\newcommand{\bbR}{\mathbb R}
\newcommand{\bbC}{\mathbb C}
\newcommand{\cU}{\mathcal U}

\newcommand{\cP}{\mathcal P}
\newcommand{\cI}{\mathcal I}
\newcommand{\cK}{\mathcal K}
\newcommand{\frD}{\mathfrak D}
\newcommand{\frX}{\mathfrak X}
\newcommand{\frY}{\mathfrak Y}

\DeclareMathOperator{\supp}{supp}

\DeclareMathOperator{\dom}{dom}
\DeclareMathOperator{\OCAT}{OCA_T}
\DeclareMathOperator{\MA}{MA}
\DeclareMathOperator{\ZFC}{ZFC}

\theoremstyle{definition}
\newtheorem{thm}{Theorem}

\newtheorem{coro}[thm]{Corollary}
\newtheorem{definition}{Definition}[section]
\theoremstyle{plain}
\newtheorem{lemma}[definition]{Lemma}
\newtheorem{proposition}[definition]{Proposition}
\newtheorem{theorem}[definition]{Theorem}
\newtheorem{corollary}[definition]{Corollary}
\theoremstyle{remark}

\newtheorem{question}[definition]{Question}
\newtheorem{claim}{Claim}

\author{Vadim Alekseev}
\address{V.A., Institute of Geometry, TU Dresden, 01062 Dresden, Germany}
\email{vadim.alekseev@tu-dresden.de}

\author{Ilijas Farah}
\address{I.F., Department of Mathematics and Statistics, York University, 4700 Keele Street, North York, Ontario, Canada, M3J 1P3, and Matematicki Institut SANU, Kneza Mihaila 36, 11 000 Beograd, p.p. 367, Serbia}
\email{ifarah@yorku.ca}

\author{Andreas Thom}
\address{A.T., Institute of Geometry, TU Dresden, 01062 Dresden, Germany}
\email{andreas.thom@tu-dresden.de}

\begin{document}

\title{Ulam stability for classes of nuclear \cstar-algebras}

\begin{abstract}
We study Ulam stability for approximate \(*\)-homomorphisms of
\(C^*\)-algebras. We prove stability results for several classes of nuclear
\(C^*\)-algebras with respect to von Neumann algebra targets, including
abelian \(C^*\)-algebras and large classes arising in the Elliott
classification program. We also discuss permanence properties, counterexamples,
and related stability phenomena. As applications, we obtain rigidity and
independence results for corona algebras.
\end{abstract}

\maketitle

\tableofcontents

 \section{Introduction}

In \cite{Ul:Problems}, S. Ulam defined an $\varepsilon$-homomorphism from a group $G$ into a metric group $(H,d)$ as a map $\varphi\colon G\to H$ that satisfies 
\[
d(\varphi(gh),\varphi(g)\varphi(h))<\varepsilon
\] 
for all $g,h \in G$, and asked under what conditions such $\varepsilon$-homomorphism can be uniformly approximated by a true homomorphism. 
This definition has been adapted to other algebraic structures equipped with a metric, giving rise to the very general and well-studied 
concept of Ulam stability. Early on, Hyers observed  that for an  $\varepsilon$-homomorphism $\varphi\colon \frX\to\frY$ between Banach spaces the sequence $2^{-n}\varphi(2^{n} x)$ converges to a  homomorphism that  $\varepsilon$-approximates~$\varphi$ (\cite{hyers1941on}, see also \cite{hyers1992approximate}). 

In case of Banach space-based structures such as \cstar-algebras one could meaningfully weaken the definition of an $\varepsilon$-homomorphism in more than one way. B. Johnson studied Ulam stability of linear maps  $\varphi$ between Banach algebras that 
are approximately multiplicative,  in the sense that 
\[
\|\varphi(ab)-\varphi(a)\varphi(b)\|<\varepsilon \|a\|\|b\|
\]
for all $a$ and $b$ in the domain 
 (\cite{johnson1988approximately}, \cite{johnson1982counterexample}), apparently unaware of Ulam's question. 
Tangentially related is the well-studied notion of weakly stable relations on a \cstar-algebra (\cite{Lo:Lifting}); it could be considered as a local-to-local analog of Ulam stability of $*$-homomorphisms, or as the  Ulam stability of equations.  The study of $\varepsilon$-\cstar-algebras (\cite{kitaev2024almost}) is subtly different from Ulam stability.  

Our notion of $\varepsilon$-$*$-homomorphisms defined below  is motivated by the study of rigidity of coronas of \cstar-algebras (\cite{Fa:All}, \cite{MV}, \cite{vignati2022rigidity}, \cite[\S 17.2]{FarahCT}, \cite{farah2025corona}); see \S\ref{S.CoronaRigidity} (also Theorem~\ref{T.CoronaEmbedding} and Theorem~\ref{T.CoronaRigidity}) for details and applications of our results.  The following definition first appeared explicitly in \cite[Def.~1.1]{MV}.

\begin{definition}
Let $A,B$ be unital $C^*$-algebras and $\varepsilon\geq0$.
An \emph{$\varepsilon$-$*$-homomorphism} is a map $\varphi\colon A \to B$ such that
for  all $a,b\in A_{\le 1}$, and all $\lambda\in\mathbb C_{\le 1}$, we have $\varphi(a)\in B_{\le 1}$ and
\begin{align*}
\|\varphi(a+b)-\varphi(a)-\varphi(b)\| &\leq \varepsilon,\\
\|\varphi(ab)-\varphi(a)\varphi(b)\| &\leq \varepsilon,\\
\|\varphi(a^*)-\varphi(a)^*\| &\leq \varepsilon,\\
\|\varphi(\lambda a)-\lambda\varphi(a)\| &\leq \varepsilon.
\end{align*}
\end{definition}

\begin{definition}
    We will denote by $\sfFD$ the class of finite-dimensional $C^*$-algebras, by $\sfAF$ the class of approximately finite-dimensional $C^*$-algebras, by $\sfUC$ the class of unital abelian $C^*$-algebras, by $\cst$ the class of all $C^*$-algebras, and by $\sfvN$ the class of von Neumann algebras. 
\end{definition}
\begin{definition}
	Suppose that $\sfC$ and $\sfD$  are classes of \cstar-algebra. Following the conventions,  we say that the pair $(\sfC,\sfD)$ is \emph{Ulam stable} (or that $\sfC$ is Ulam stable with respect to $\sfD$) if there is a function $f_{\sfC,\sfD}(\varepsilon)\to 0$ as $\varepsilon\to 0$ such that for every $A\in \sfC$, $B\in \sfD$, and $\varepsilon$-$*$-homomorphism $\varphi\colon A\to B$ there is a $*$-homomorphism $\psi\colon A\to B$ such that $\sup_{\|a\|\le 1}\|\varphi(a)-\psi(a)\|\le f_{\sfC,\sfD}(\varepsilon)$.
We will write $f_{\sfC}$ for $f_{\sfC,\sfvN}$. 
\end{definition}

The following summarizes existing results on Ulam stability. 
\begin{theorem}\label{thm:MV_AF} Each of the following pairs is Ulam stable. 
\begin{enumerate}
	\item $(\sfFD,\sfFD)$, by the second author \cite{Fa:All}. 
	\item \label{I.MV_fin} $(\sfFD, \cst)$, by McKenney--Vignati \cite{MV}. 
	\item \label{I.MV_AF} $(\sfAF,\sfvN)$, by McKenney--Vignati \cite{MV}. 
	\item \label{I.Semrl} $(\sfUC,\sfUC)$, by \v Semrl \cite{semrl1999non}. \qed 
	\end{enumerate}
\end{theorem}

The implication from \eqref{I.MV_fin} to \eqref{I.MV_AF} is a consequence of a result of B. Johnson (see Theorem~\ref{T.Closure}\eqref{thm:Johnson}). 
The result of \eqref{I.Semrl} is a consequence of a more general theorem. In \cite[Theorem 2.1]{semrl1999non} it was proved that approximate homomorphisms (not necessarily $*$-homomorphisms, and with a different definition of `approximate') from an arbitrary Banach algebra into a Banach algebra of the form $C(X)$ for a compact Hausdorff $X$ are Ulam-stable. This implies Ulam stability for approximate $*$-homomorphisms between unital abelian \cstar-algebras (see \cite[Theorem 5.18]{farah2025corona}).

The aim of this note is to extend Theorem~\ref{thm:MV_AF}(\ref{I.MV_AF}) to a much larger class of nuclear \cstar-algebras and give applications to corona rigidity, but we first give an alternative proof of Theorem \ref{thm:MV_AF}(\ref{I.Semrl}) in Theorem~\ref{T.Abelian+}. The case when 
 the range is $\bbC$ (which is the key step in our proof) has been proved independently by Jennifer Pi, by using completely different methods. 

After proving various closure properties of  classes of \cstar-algebras that are Ulam stable with respect to von Neumann algebras, 
using the recent breakthrough results in the classification of nuclear \cstar-algebras (\cite{gong2020classificationI}, \cite{gong2020classificationII}, also \cite{carrion2023classifying}), we prove the following (it is a consequence of Theorem~\ref{T.Classifiable}). 

\begin{thm}\label{T.Elliott}
	The class of all stably finite Elliott-classifiable unital \cstar-algebras is Ulam stable with respect to von Neumann algebras. 
\end{thm}

The following is proved as Corollary~\ref{T.AH}. 

\begin{thm}\label{T.C(X)otimesF}
The class of all inductive limits of \cstar-algebras of the form $C(X)\otimes F$ for a compact Hausdorff space $X$ and an $\sfAF$ algebra $F$ is Ulam stable with respect to the von Neumann algebras. 
\end{thm}

While this does not cover all AH \cstar-algebras (approximately homogeneous \cstar-algebras), it covers many counterexamples to Elliott's conjecture, such as all Villadsen algebras of the first type (\cite{Vill:Simple}, see also \cite[\S 3.1]{toms2009Elliott}) including Toms's counterexample to the Elliott conjecture (\cite[Theorem~1.1]{To:On}; see the beginning of Section 3).

While the most interesting instance of these theorems is the case of separable \cstar-algebras, it should be noted that separability is not required. This is a consequence of a general fact proved by using a standard reflection argument discussed in \S\ref{S.Reflection}.

Our initial motivation comes from the rigidity theory of coronas of separable \cstar-algebras under set-theoretic assumptions. 
More details and proof of the Theorem~\ref{T.CoronaRigidity} below(stated as Theorem~\ref{T.CoronaRigidity+}) will be given in \S\ref{S.CoronaRigidity}. Forcing axioms $\OCAT$ and $\MA$ will not be used explicitly, instead we will rely on results from \cite{mckenney2020forcing} and \cite{vignati2022rigidity}. What matters for us is that they are axioms relatively consistent with $\ZFC$ that provide an environment for strongest possibly rigidity of quotient structures in several categories; see \cite{farah2025corona} for more details.  As usual, for a \cstar-algebra $A$ we denote its multiplier algebra by $\cM(A)$, its corona $\cM(A)/A$ by $\cQ(A)$, and write
\[
A_\infty=\ell_\infty(A)/c_0(A). 
\]

\begin{thm} \label{T.CoronaRigidity} Assume $\OCAT$ and $\MA$. 
	If $A_n$, $B_n$, for $n\in \bbN$, are Elliott-classifiable unital \cstar-algebras,  then the following are equivalent. 
	\begin{enumerate}
		\item The coronas of $\bigoplus_n A_n$ and $\bigoplus_n B_n$ are isomorphic. 
		\item There is a bijection $\pi$ between cofinite subsets of $\bbN$ such that $A_{\pi(n)}\cong B_n$ for all $n\in \dom(\pi)$. 
	\end{enumerate}
\end{thm}

Some special cases of Theorem~\ref{T.CoronaRigidity} were proved in \cite{vignati2022rigidity} by using different methods.  

\medskip

In \cite[Theorem~D]{vignati2022rigidity} it was proved that $\OCAT$ and $\MA$ imply that if $A$ is unital, separable, infinite-dimensional \cstar-algebra then  $A_\infty$ does not embed into the Calkin algebra $\cQ(H)$, and that if $A$ is not stably finite then $A_\infty$ does not embed into  the corona of any stable, stably finite \cstar-algebra. 
The following finer result is proved as Theorem~\ref{T.CoronaEmbedding+}. 

\begin{thm}\label{T.CoronaEmbedding} Assume $\OCAT$ and $\MA$. 
	Suppose that $A$ and $B$ are separable, unital, \cstar-algebras and that $A$ is stably finite and Elliott classifiable,  then the following are equivalent. 
	\begin{enumerate}
		\item $A_\infty$ embeds into $\cQ(B\otimes \cK)$ (not necessarily unitally). 
		\item $A$ embeds into $B\otimes \cK$. 
	\end{enumerate}
\end{thm}

Already the instance of this theorem when $A$ and $B$ are UHF algebras gives numerous non-embedding results. This instance of the theorem follows already by combining Vignati's methods from \cite{vignati2022rigidity} with \cite{CSSWW:Perturbations} (see \S\ref{S.KadisonKastler} for relevance of the latter), but it has apparently been overlooked. 
By combining these results with  \cite{ghasemi2014reduced} and \cite{farah2017calkin} we obtain the following, proved as Corollary~\ref{C.Independence.+}

\begin{coro}\label{C.Independence}
\begin{enumerate}
\item[]
\item The assertion that $(M_{2^\infty})_\infty$ embeds into $\cQ(M_{3^\infty}\otimes \cK(H))$ is independent from ZFC. 
\item 	There are UHF \cstar-algebras $A_n$, $B_n$, for $n\in \bbN$ such that the assertion that $\prod_n A_n/\bigoplus_n A_n$ and $\prod_n B_n/\bigoplus_n B_n$ are isomorphic is independent from ZFC. 
\end{enumerate}
\end{coro}

Counterexamples to Ulam stability were given by B. Johnson (\cite{johnson1982counterexample}), who proved, when expressed in our terminology, that the pair $(\{c_0\}, \{C([0,1],\cK(H))\})$ is not Ulam stable. Since his terminology is different from ours, this is explained in \S\ref{S.Johnson}. 

Much of the study of corona rigidity concentrated on finding the correct notion of a trivial isomorphism and proving that no other isomorphisms can be constructed without using additional set-theoretic axioms such as the Continuum Hypothesis, CH (see the discussion on the relevance of CH in \cite[\S 6 and, on a very abstract level, \S 12.1]{farah2025corona}, also the introduction to \cite{vignati2022rigidity}). This is achieved by showing that the forcing axioms $\OCAT$ and $\MA$ together imply that all isomorphisms are of this form. 
As per \cite[Definition~1.2]{vignati2022rigidity}, an isomorphism $\Phi$ between $\cQ(A)$ and $\cQ(B)$ is \emph{algebraically trivial} if there are $a\in A_+$, $b\in B_+$, and an isomorphism $\varphi$ between the hereditary \cstar-algebras $\overline{(1-a)A(1-a)}$ and $\overline{(1-b)B(1-b)}$ such that $\varphi$ extends to an isomorphism $\Phi$ between $\overline{(1-a)\cM(A)(1-a)}$ and $\overline{(1-b)\cM(B)(1-b)}$ which satisfies  $\Phi(x+A)=\varphi(x)+B$ for all $x\in 
\overline{(1-a)\cM(A)(1-a)}$. 
Surprisingly, by \cite[Theorem~C]{vignati2022rigidity}, forcing axioms $\OCAT$ and $\MA$ imply that every isomorphism between coronas of abelian \cstar-algebras is algebraically trivial.

We use Johnson's results to refute 
\cite[Conjecture 5.1]{vignati2022rigidity}, where it was conjectured that  $\OCAT$ and $\MA$ together imply that all isomorphisms between all coronas of separable \cstar-algebras are algebraically trivial. The following is proved as Theorem~\ref{T.Nontrivial+}. 

\begin{thm} \label{T.Nontrivial} Let $B$ denote the unitization of $C([0,1],\cK(H))$. Then $B_\infty$ has an automorphism that is not algebraically trivial.
\end{thm}

We consider Theorem~\ref{T.Nontrivial} and its proof as an indication that the definition of `trivial isomorphism' needs to be revised rather than as a mere counterexample. 

All \cstar-algebras for which we can prove Ulam stability with respect to von Neumann algebras are nuclear and stably finite. On the other hand, the possibility that the class of all \cstar-algebras is Ulam stable with respect to $\sfvN$ is not ruled out by known results. 
In \S\ref{S.B(H)} we prove that the pair $(\{B(H)\},\{B(H)\})$ is Ulam stable. 

\medskip
The paper is organized as follows. In \S\ref{S.General} we find a convenient reformulation of Ulam stability of a pair $(\sfA,\sfB)$. In \S\ref{S.Groups} we show how Ulam stability for \cstar-algebras follows from Ulam stability of their unitary groups, considered as discrete groups. Closure properties of Ulam stable pairs are given in \S\ref{S.SimpleClosure} and \S\ref{S.Extensions}.  Our results are extended to nonseparable \cstar-algebras in \S\ref{S.Reflection}.  In~\S\ref{S.Johnson} we rehash B. Johnson's examples of failure of Ulam stability for \cstar-algebras. 
In \S\ref{S.B(H)} we prove that the pair $(\{B(H)\}, \{B(H)\})$ is Ulam stable.  In~\S\ref{S.KadisonKastler} we use results on Kadison--Kastler stability to prove that (using additional assumptions on~$A$) for sufficiently small $\varepsilon>0$ the existence of an $\varepsilon$-injective $\varepsilon$-$*$-homomorphism from $A$ to~$B$ implies $A$ is isomorphic to a \cstar-subalgebra of~$B$ and the existence of an $\varepsilon$-$*$-isomorphism between $A$ and~$B$ implies isomorphism. Finally, applications to corona rigidity are given in~\S\ref{S.CoronaRigidity}. 	

\section{General results} 
\label{S.General}

The main result of this section, Corollary~\ref{C.UlamStableCorrected}, gives a convenient reformulation of Ulam stability. It is a consequence of Lemma~\ref{L.Corrected} which shows  that every $\varepsilon$-$*$-homomorphism can be uniformly approximated by a map with additional regularity properties. 

\medskip

Throughout, for $a,b\in A$ and $\varepsilon>0$ we write $a\approx_\varepsilon b$ if $\|a-b\|\leq \varepsilon$. By $A_1$ we denote the unit ball of the \cstar-algebra $A$.

\begin{lemma}\label{L.approx.continuity} 
	Suppose that $\varepsilon>0$ and $\varphi\colon A\to B$ is an $\varepsilon$-$*$-homomorphism, such that $\varphi(A_{\textrm{sa}})\subseteq B_{\textrm{sa}}$. Then the following holds. 
	\begin{enumerate}
		\item 	 \label{1.approx} Every $a\in A_1$ satisfies $\|\varphi(a)\|\leq \|a\|+\varepsilon \|a\|$. 
		\item \label{2.approx} All $0\leq a\leq b\leq 1$ satisfy $\varphi(a)\leq \varphi(b)+3\varepsilon$.
	\end{enumerate}
\end{lemma}

\begin{proof} \eqref{1.approx} If $a\neq 0$ then 
	$\|\varphi(a/\|a\|)-\frac 1{\|a\|} \varphi(a)\|\leq \varepsilon$. Since $\|\varphi(a/\|a\|)\|\leq 1$, $\|\varphi(a)\|\leq \|a\| (1+\varepsilon)$ follows. Since 
	$\|\varphi(0)-2\varphi(0)\|\leq \varepsilon$, the case when $a=0$ follows. 
	
	\eqref{2.approx} Let $c$ be such that $a+c^*c=b$. Then $\varphi(b)-\varphi(a)\approx_\varepsilon \varphi(c^*c)\approx_{2\varepsilon}\varphi(c)^*\varphi(c)\geq 0$ and the conclusion follows. 
	\end{proof}

\begin{lemma}
	\label{L.exp} For all self-adjoint $a$ and $b$ with $\max(\|a\|,\|b\|)\leq 1$  and sufficiently small $\varepsilon>0$ we have that $\sup_{|t|\leq 1} \|\exp(ita)-\exp(itb)\|\leq \varepsilon$ implies $\|a-b\|\leq 2\sqrt \varepsilon$.  \end{lemma}
\begin{proof}
	Set
	\[
	g(t)=\exp(ita)-\exp(itb), \qquad t\in[-1,1].
	\]
	Then \(g(0)=0\), and by assumption
	\[
	\sup_{|t|\leq 1}\|g(t)\|\leq \varepsilon .
	\]
	Moreover,
	$
	g'(t)=ia\exp(ita)-ib\exp(itb),
	$
	so that
	$
	g'(0)=i(a-b).
	$
	Also
	\[
	g''(t)=-a^2\exp(ita)+b^2\exp(itb).
	\]
	Since \(a\) and \(b\) are self-adjoint contractions, \(\|\exp(ita)\|=\|\exp(itb)\|=1\), hence
	\[
	\|g''(t)\|\leq \|a^2\|+\|b^2\|\leq 2
	\]
	for all \(t\in[-1,1]\).
	For \(0<h\leq 1\), Taylor's formula with integral remainder gives
	\[
	g(h)=h g'(0)+\int_0^h (h-s)g''(s)\,ds.
	\]
	Therefore, 
	\[
	h\|g'(0)\|
	\leq \|g(h)\|+\int_0^h (h-s)\|g''(s)\|\,ds
	\leq \varepsilon+h^2.
	\]
	Thus
	\[
	\|a-b\|=\|g'(0)\|\leq \frac{\varepsilon}{h}+h.
	\]
	Choosing \(h=\sqrt{\varepsilon}\) for \(0<\varepsilon\leq 1\), we obtain
	$
	\|a-b\|\leq 2\sqrt{\varepsilon}.
	$
	This proves the claim.
\end{proof}

\begin{lemma}\label{L.g}
	There exists a function $g(\varepsilon)\to 0$ as $\varepsilon\to 0$ with the following property. For all \cstar-algebras $A$ and $B$ and every $\varepsilon$-$*$-homomorphism $\varphi\colon A\to B$ such that $\|a\|\leq 1$ implies $\|\varphi(a)\|\leq 1$ and $\varphi[A_{\textrm{sa}}]\subseteq B_{\textrm{sa}}$ the following 
	holds, where $\log$ denotes the main branch of the natural logarithm:
	\begin{enumerate}
		\item \label{I.exp} $\exp(i\varphi(a))\approx_{g(\varepsilon )}\varphi(\exp(i a))$ for all $a\in A_{\textrm{sa}}$ with $\|a\|\leq 1$. 
		\item $\varphi(a)\approx_{g(\varepsilon)} -i \log \varphi(\exp(ia))$ for all $a\in A_{\textrm{sa}}$ with $\|a\|\leq 1$. \qed 
	\end{enumerate}
\end{lemma}

\begin{proof}
	We give a soft proof of \eqref{I.exp}. 
	Assume otherwise, that there is $\delta>0$ and for every~$n$ there are a $1/n$-$*$-homomorphism $\varphi_n\colon A_n\to B_n$ and a self-adjoint contraction $a_n$ in $A_n$ such that 
	$\|\exp(i\varphi_n(a_n))-\varphi_n(\exp(i a_n))\|>\delta$
	
	Fix a nonprincipal ultrafilter $\cU$ on $\bbN$ and consider ultraproducts  $A=\prod_\cU A_n$ and $B=\prod_\cU B_n$. Then the function from $\prod_n A_n$ to $\prod_n B_n$ defined by  $\varphi((a_n))=(\varphi_n(a_n))$ lifts a $*$-homomorphism from $A$ into $B$, and 
	the element $a$ with representing sequence $(a_n)$ is a self-adjoint contraction in $A$ such that $\|\exp(i\varphi(a))-\varphi(\exp(i a))\|>\delta$ contradiction. 
	
	Soft proofs of the remaining statement is analogous and therefore omitted. 
\end{proof}

\begin{lemma}\label{L.Corrected}
		There exists a function $g(\varepsilon)\to 0$ as $\varepsilon\to 0$ with the following property. For all \cstar-algebras $A$ and $B$ and every $\varepsilon$-$*$-homomorphism $\varphi\colon A\to B$ 
 there exists a $g(\varepsilon)$-$*$-homomorphism $\psi\colon A\to B$ such that $\sup_{\|a\|\leq 1} \|\psi(a)-\varphi(a)\| \leq g(\varepsilon)$ and in addition $\psi$ satisfies the following for all $a\in A$ such that  $\|a\|\leq 1$. 
\begin{enumerate}
	\item\label{a.Corrected.norm} $\|\psi(a)\|\leq 1$. 
	\item \label{a.Corrected.*} $\psi(a^*)=\psi(a)^*$. 
	\item \label{a.Corrected.scalar} $\psi(\lambda a)=\lambda \psi(a)$ for all $\lambda\in \bbC$ such that $|\lambda|\leq 1$. 
	\item \label{a.Corrected.ctns}$\|\psi(a)\|\leq \|a\|$. 
	\item \label{a.Corrected.leq} $\psi(a)\leq \psi(b)+3\varepsilon$ for all $0\leq a\leq b\leq 1$. 
			\item \label{a.Corrected.exp} $\exp(i\psi(a))\approx_{g(\varepsilon )}\psi(\exp(i a))$ for all $a\in A_{\textrm{sa}}$ with $\|a\|\leq 1$. 
	\item \label{a.Corrected.log} $\psi(a)\approx_{g(\varepsilon)} -i \log \psi(\exp(ia))$ for all $a\in A_{\textrm{sa}}$ with $\|a\|\leq 1$.
\end{enumerate}
\end{lemma}

\begin{proof}
	Let $\varphi\colon A\to B$ be an $\varepsilon$-$*$-homomorphism. Define $\varphi_1\colon A\to B$ by 
	\[
	\varphi_1(a)=\frac 12 (\varphi(a)+\varphi(a^*)^*). 
	\]
	Then for $\|a\|\leq 1$ we have $\|\varphi_1(a)-\varphi(a)\|\leq \frac 12 \|\varphi(a^*)^*-\varphi(a)\|\leq \frac 12\varepsilon$. 
	
	As in \cite[Theorem 5.18]{farah2025corona}, 
	define $\varphi_2\colon A\to B$ by $\varphi_2(0)=0$ and for $a\neq 0$ let 
	\[
	\varphi_2(a)=\|a\|\varphi_1(a/\|a\|). 
	\]
	Then $\|\varphi_2(0)-\varphi_1(0)\|=\|\varphi_1(0)\|\leq \varepsilon$ by Lemma~\ref{L.approx.continuity} while for $a\neq 0$ and $\|a\|\leq 1$ we have  $\|\varphi_2(a)-\varphi_1(a)\|\leq \varepsilon$. 
	
	Then $\varphi_2$ maps $A_1$ into $B_1$ and it satisfies \eqref{a.Corrected.norm}, \eqref{a.Corrected.*}, and \eqref{a.Corrected.scalar}. 
	Since $\|a\|\leq 1$ and $\|\varphi(a)\|\leq 1$, \eqref{a.Corrected.scalar} implies $\|\varphi(a)\|\leq \|a\|$ for $\|a\|\leq 1$, hence   \eqref{a.Corrected.ctns} holds. \eqref{a.Corrected.leq} is Lemma~\ref{L.approx.continuity}. 
	
	Since it uniformly approximates $\varphi$ on the unit ball of $A$, 
	it is straightforward to verify that $\varphi_2$ is an $O(\varepsilon)$-$*$-homomorphism, 
		This function satisfies \eqref{a.Corrected.exp} and \eqref{a.Corrected.log} by Lemma~\ref{L.g}. 
\end{proof}

We now have a convenient reformulation of Ulam stability for classes of \cstar-algebras. 

\begin{corollary}\label{C.UlamStableCorrected}
Suppose that $\sfC$ and $\sfD$  are classes of \cstar-algebra. Then the following are equivalent. 
\begin{enumerate}
	\item\label{1.C.Ulam.Corrected} The pair $(\sfC,\sfD)$ is Ulam stable. 
\item\label{2.C.Ulam.Corrected} There is a function $f(\varepsilon)\to 0$ as $\varepsilon\to 0$ such that for every $A\in \sfC$, $B\in \sfD$, and $\varepsilon$-$*$-homomorphism $\varphi\colon A\to B$  that for all $a,b$ in $A_1$ and all $\lambda\in \bbC_1$ satisfies
\begin{enumerate}[label=(C\arabic*)]
	\item\label{Corrected.norm} $\|\varphi(a)\|\leq 1$.
	\item \label{Corrected.*} $\varphi(a^*)=\varphi(a)^*$. 
	\item \label{Corrected.scalar} $\varphi(\lambda a)=\lambda \varphi(a)$ for all $\lambda\in \bbC$ such that $|\lambda|\leq 1$. 
	\item \label{Corrected.ctns}$\|\varphi(a)\|\leq \|a\|$. 
	\item \label{Corrected.leq} $\varphi(a)\leq \varphi(b)+3\varepsilon$ for all $0\leq a\leq b\leq 1$. 
	\item \label{Corrected.exp} $\exp(i\varphi(a))\approx_{g(\varepsilon )}\varphi(\exp(i a))$ for all $a\in A_{\textrm{sa}}$ with $\|a\|\leq 1$. 
	\item \label{Corrected.log} $\varphi(a)\approx_{g(\varepsilon)} -i \log \varphi(\exp(ia))$ for all $a\in A_{\textrm{sa}}$ with $\|a\|\leq 1$.
\end{enumerate}
%
there is a $*$-homomorphism $\psi\colon A\to B$ such that $$\sup_{\|a\|\le 1}\|\varphi(a)-\psi(a)\|\le f(\varepsilon).$$
\end{enumerate}
\end{corollary}
\begin{proof}
	By Lemma~\ref{L.Corrected}, every $\varepsilon$-$*$-homomorphism $\varphi_0$ from $A$ into $B$ can be $g(\varepsilon)$-approximated on $A_1$ by a function $\varphi$ that satisfies  conditions in \eqref{2.C.Ulam.Corrected}. Clearly $\varphi$ is uniformly approximated by a $*$-homomorphism on $A_1$ if and only if $\varphi_0$ is uniformly approximated by a $*$-homomorphism on $A_1$. 
\end{proof}

\section{Ulam stability -- from groups to \cstar-algebras}\label{S.Groups}

The main result of this section, Theorem~\ref{T.Ulam2}, will be used together with  Ulam stability for amenable groups, see \cite{BMT,Kaz:epsilon}, towards proofs of our main results. 
Note that the group homomorphism $\pi$ is not required to be continuous.

\begin{theorem}\label{T.Ulam2}
	There exists a function $h(\varepsilon)\to 0$ as $\varepsilon\to 0$ with the following property. 
	For every unital  \cstar-algebra $A$, unital \cstar-algebra $B$,  and an $\varepsilon$-$*$-homomorphism $\varphi\colon A \to B$, if there exists a group homomorphism $\pi\colon U(A)\to U(B)$ such that 
	\[
	\sup_{u\in U(A)} \|\varphi(u)-\pi(u)\|\leq \varepsilon+120\varepsilon^2
	\] 
	then  there exists a $*$-homomorphism $\Phi\colon A\to B$ such that $\Phi(\exp(ia))=\pi(\exp(ia))$ for all $a\in A_{\textrm{sa}}$ and 
	$\sup_{\|a\|\le 1}\|\varphi(a)-\Phi(a)\|\le h(\varepsilon)$.
\end{theorem}

The methods used in the proof have Lie-algebraic flavour and are related to those of \cite[Theorem~1.1]{fogang2024distinguishing}  where it was proved
that for (complex) \cstar-algebras $A$ and $B$, 
the existence of an isomorphism $\pi$ between  the Banach Lie groups  $\SU_0(A)$ and $\SU_0(B)$ as topological groups is equivalent to the existence of an isomorphism $\psi$ between $A$ and $B$ as real \cstar-algebras, and moreover that $\psi$ can be chosen so that it extends $\pi$. See also \cite{ando2025Lie}. In hindsight, the Ulam stability results of \cite{MV} and \cite{Fa:All} implicitly used Lie-algebraic methods.

Lemma~\ref{L.pi-to-rho-and-Phi} below should be compared with the method used in \cite{KanRe:Ulam}, where it was proved (among other things) that an approximate group homomorphism between certain Boolean algebras of the form $\cP(\bbN)/\cI$  that is (in a certain precise way) an approximate Boolean algebra homomorphism  is near a Boolean algebra homomorphism. 
The assumptions of this lemma will be secured by \cite[Theorem~2.1]{BMT}. 

While $\varphi$ in the following lemma  is an arbitrary function, the lemma will be applied in the proof of Lemma~\ref{L.pi-to-rho-and-Phi} in case when $\varphi$ is an approximate $*$-homomorphism. 

\begin{lemma}\label{L.FunctionfIsContinuous}
	Suppose that $X$ and $Y$ are locally compact Hausdorff spaces,  $\varphi\colon C_0(X)\to C_0(Y)$ and $f\colon Y\to X$ be functions such that $|\varphi(a)(y)-a(f(y))|<1/2$ for all $a\in C_0(X)$.  Then $\psi(a)(y)=a(f(y))$ defines a $*$-homomorphism from $C_0(X)$ to $C_0(Y)$ such that $\|\psi(a)-\varphi(a)\|\leq 1/2$ for all $a\in C_0(X)$. 
\end{lemma}

\begin{proof}
	We claim that the function $f\colon Y\to X$ is continuous. 
	Otherwise,  there is a net $(y_\lambda)$ in $Y$ such that $y=\lim_\lambda y_\lambda$ exists but $\lim_\lambda f(x_\lambda)$ either does not exist or it exists but is not equal to $f(y)$. In either case, there are an open neighbourhood $U$ of $f(y)$ and a subnet $(y'_{\lambda})$ of $(y_\lambda)$ such that $f(y'_{\lambda})\notin U$ for all $\lambda$. By the Tietze extension theorem, there is $a\in C_0(X)$ such that $0\leq a\leq 1$, $\supp(a)\subseteq U$, and $a(f(y))=1$.  Then $b=\varphi(a)$ satisfies  $|b(z) -a(f(z))|<1/2$ for all $z\in Y$. In particular $b(y'_{\lambda})<1/2$ for all $\lambda$ and $b(y)>1/2$, hence $b$ cannot be continuous; contradiction. 
	
	Therefore, $f\colon Y\to X$ is continuous and by the Gelfand--Naimark duality it determines a $*$-homomorphism $\psi\colon C_0(X)\to C_0(Y)$ with the required properties. 
\end{proof}

\begin{lemma} \label{L.pi-to-rho-and-Phi} 
	There exists a function $h(\varepsilon)\to 0$ as $\varepsilon\to 0$ with the following property for sufficiently small $\varepsilon$. 
	Suppose that $X$ is compact and Hausdorff and that $\pi\colon U(C(X))\to U(B)$ is a group homomorphism such that some $\varepsilon$-$*$-homomorphism $\varphi\colon C(X)\to B$ satisfies the following.  
	\begin{enumerate}
		\item If $x\in C(X)$ and $\|x\|\leq 1$ then $\|\varphi(x)\|\leq 1$,
		\item $\varphi[C(X)_{\textrm{sa}}]\subseteq B_{\textrm{sa}}$,   $\varphi(0)=0$, and $\varphi(1)=\pi(1)=1$, 
		\item $\sup_{u\in U(C(X))} \|\varphi(u)-\pi(u)\|\leq \varepsilon+120\varepsilon^2$.   
	\end{enumerate} 
	Then there is a $*$-homomorphism $\Phi\colon C(X)\to B$ such that  $\Phi(u)=\pi(u)$ for all $u\in U(C(X))$ and 
	\(
	\sup_{\|a\|\leq 1} \|\Phi(a)-\varphi(a)\|\leq h(\varepsilon)\). 
\end{lemma}

\begin{proof}
	With $g$ as in Lemma~\ref{L.g},  let 
	\[
	h(\varepsilon)=\max(4g(\varepsilon),  g(\varepsilon)+\varepsilon+120\varepsilon^2, 4 (g(\varepsilon)+2\varepsilon+120\varepsilon^2)^{1/2}).
	\] 
	Fix $\varepsilon>0$ such that    $2 \varepsilon+120\varepsilon^2<1$, $\frac 32 \varepsilon+120\varepsilon^2<1$, and $h(\varepsilon)+\varepsilon+120\varepsilon^2<\sqrt 2$. 
	
	Fix $C(X)$, $B$, a group homomorphism  $\pi\colon U(C(X))\to B$ and an $\varepsilon$-$*$-homomorphism $\varphi\colon C(X)\to B$ as in the statement of this lemma.  
	We will prove that 
	there is a continuous group homomorphism $\rho\colon C(X)_{\textrm{sa}}\to B_{\textrm{sa}}$ such that $		\pi(\exp(ia))=\exp(i\rho(a))$
		for all $a$ in $C(X)_{\textrm{sa}}$.
	
	We first prove the existence of $\rho$ and the conclusion of the lemma under the additional assumption that $B=\bbC$. 
In this case, 	\begin{equation}\label{eq.a->pi(exp(ia))}
		a\mapsto \pi(\exp(ia)) 
	\end{equation}
	is a group homomorphism from $C(X)_{\textrm{sa}}$ into $\bbT$. 
	We claim that it is continuous. 
	Otherwise, there are $a_n\in C(X)$ such that $\lim_n a_n=0$ but $\lim_n \|\pi(\exp(ia_n))-1\|>0$. By compactness of $\bbT$, some $\lambda\in \bbT\setminus \{1\}$ and a subsequence (still denoted $a_n$) satisfy $\lim_n \pi(\exp(ia_n))=\lambda$. Since $\varepsilon+120\varepsilon^2<1<\sqrt 3$, there exists $m\geq 1$ such that $|\lambda^m-1|>2\varepsilon+120\varepsilon^2$. Thus all sufficiently large $n$ satisfy 
	\[
	\|\pi(\exp(i m a_n))-1\|>2\varepsilon+120\varepsilon^2.
	\]
	Lemma~\ref{L.approx.continuity} and $\varphi(1)=1$ together imply  that $\limsup_n \|\varphi(\exp(i m a_n))-1\|\leq \varepsilon$. Since $\pi(u)\approx_{\varepsilon+120\varepsilon^2} \varphi(u)$ for all $u\in U(C(X))$,  we have 
	$\limsup_n \|\pi(\exp(i m a_n))-1\|\leq 2\varepsilon+120\varepsilon^2$; contradiction.
	
	Fix $a\in C(X)_{\textrm{sa}}$. Then $t\mapsto \pi(\exp(ita))$ is a continuous homomorphism from $\bbR$ into $\bbT$. Since the Pontryagin dual of $\bbR$ is $\bbR$, there exists $\rho(a)\in \bbR$ such that 
	\[
	\pi(\exp(ita))=\exp(it \rho(a))
	\]
	for all $t\in \bbR$. 
	Clearly $\rho(sa+tb)=s\rho(a)+t\rho(b)$ for all $a,b$ in $C(X)_{\textrm{sa}}$ and $s,t$ in $\bbR$.  
	Since the homomorphism in \eqref{eq.a->pi(exp(ia))} is  continuous, so is $\rho$. 
	Thus, $\rho\colon C(X)_{\textrm{sa}}\to \bbC$ is a continuous linear functional. 
	By the Riesz representation theorem, there is a signed Borel measure $\mu$ on $X$ such that 
	\[
	\rho(a)=\int a\, d\mu
	\]
	for all $a$. We will prove that $\mu$ is a point-mass measure for some $x\in X$. 
	If the support of~$\mu$ has more than one element, then $X$ has two disjoint, $\mu$-positive,  open subsets $U$~and~$V$. Let $a$ and $b$ be such that $\supp(a)\subseteq U$, $\int a\,d\mu\neq 0$, $\supp(b)\subseteq V$, and $\int b\, d\mu\neq 0$. This is absurd by the following claim. 
	
	\begin{claim}
		For all $a$ and $b$ in $C(X)_{\textrm{sa}}$, if $ab=0$ then $\rho(a)\rho(b)=0$. 
	\end{claim}
	
	\begin{proof}
		Assume otherwise. Since $\rho$ is linear we may assume that $\rho(a)=\rho(b)=\pi$. (The norms of such $a$ and $b$ may be very large but that's fine.)
		Since $ab=0$, with $u=\exp(ia)$ and $v=\exp(ib)$ we have 
		\[
		(u-1)(v-1)=0. 
		\]
		Using $\varphi(0)=0$, $\|\varphi(u)\|\leq 1$ and $\|\varphi(v)\|\leq 1$, we have 
		\[
		0\approx_{4\varepsilon} (\varphi(u)-1)(\varphi(v)-1)\approx_{4 (\varepsilon+120\varepsilon^2)}
		(\pi(u)-1)(\pi(v)-1)=4.
		\]
		This contradicts the assumption that 
		$2 \varepsilon+120\varepsilon^2<1$. 
	\end{proof}
	
	Therefore $\mu$ is a point-mass measure, and there exist $x\in X$ and $r\in \bbR$ such that $\rho(a)=ra(x)$ for all $a\in C(X)_{\textrm{sa}}$. 
	
	\begin{claim}
		With $x$ and $r$ as in the previous paragraph we have $r=1$, hence $\rho(a)=a(x)$.  
	\end{claim}
	
	\begin{proof}
		Assume otherwise, that $r\neq 1$. Then 
		\[
		\sup_{s\in \bbR} |\exp(is)-\exp(irs)|=
		\sup_{s\in \bbR} |\exp(i(1-r)s)-1|=2. 
		\]
		However, for each $s\in \bbR$  we have  $\phi(\exp(is 1_{C(X)}))\approx_\varepsilon \exp(is)\varphi(1_{C(X)})=\exp(is)$. 
		On the other hand, $\pi(\exp(is 1_{C(X)}))=\exp(irs)$, contradicting $\pi(u)\approx_{\varepsilon+120\varepsilon^2}\varphi(u)$ for all unitaries~$u$ and the displayed formula. 
	\end{proof}
	%

	Fix $a\in C(X)_{\textrm{sa}}$ with $\|a\|\leq 1$. Then for all $|t|\leq 1$, with $g(\varepsilon)$ as in Lemma~\ref{L.g}  and $x$ such that $\rho(a)=a(x)$ we have 
	\[
	\exp(it \varphi(a))\approx_{g(\varepsilon)} \varphi(\exp (ita))\approx_{\varepsilon+120\varepsilon^2} \pi(\exp(ita))=\exp(it\rho(a))=\exp(it  a(x)). 
	\]
	By our assumptions $\|\varphi(a)\|\leq 1$,  thus 
	Lemma~\ref{L.exp} implies 
	\[
	\|\varphi(a)-a(x)\|\leq 2(g(\varepsilon)+2\varepsilon+120\varepsilon^2)^{1/2}, 
	\] 
	and the right-hand side of this inequality is not greater than 
	$h(\varepsilon)/2$. Therefore we have   $\varphi(a)\approx_{h(\varepsilon)/2}\rho(a)$.
	
	Let $\Phi(c)=c(x)$ for $c\in C(X)$. It is a $*$-homomorphism and $\Phi(a+ib)=\rho(a)+i\rho(b)$ for self-adjoint $a$ and $b$. 
	Hence $\Phi(c)\approx_{h(\varepsilon)} \varphi(c)$ for $\|c\|\leq 1$. 
	It remains to prove that $\Phi(u)=\pi(u)$ for all $u\in U(C(X))$. Assume otherwise, and fix $u$ for which this fails. By replacing $u$ with $u\pi(u)^{-1}$, we may assume $\pi(u)=1$ and $\Phi(u)\neq 1$. Let $m$ be such that $|\Phi(u)^m-1|>\sqrt 2$. 
	Then we have 
	\[
	1=\pi(u^m)\approx_{\varepsilon+120\varepsilon^2} \varphi(u^m)\approx_{h(\varepsilon)} \Phi(u^m),
	\]
	contradicting $\varepsilon+120\varepsilon^2+h(\varepsilon)<\sqrt 2$. 
	We therefore have $\pi(u)=\Phi(u)$ for all $u$,  
	completing the proof of Lemma  in case when $B=\bbC$. 
	
\medskip

	Now we consider the general case.  Since $\pi[U(C(X))]$ is an abelian group, the \cstar-algebra generated by it is isomorphic to $C(Y)$ for some compact Hausdorff space $Y$. Moreover, every contraction $f \in C(X)$ is a sum of four unitaries $f = u_1 + u_2 + u_3 + u_4$ in a canonical way, and this allows us to replace $\varphi(f)$ by $\pi(u_1) + \pi(u_2) + \pi(u_3) + \pi(u_4)$. This does change the image only up to $O(\varepsilon)$ and hence, we may assume without loss of generality that $\varphi$ takes values in $C(Y).$ 
    Thus, we have $\pi\colon U(C(X))\to U(C(Y))$ and $\varphi \colon C(X)\to C(Y)$ as in the statement of the lemma. 
    
	For each $y\in Y$ consider the evaluation function $\Eval_y\colon C(Y)\to \bbC$. 
	Then $\Eval_y \circ \pi$ and $\Eval_y \circ \varphi$ satisfy the assumptions of this lemma and by the already proved instances of this Lemma   there exists $x=f(y)$ in $X$ such that $\varphi_y(a)=a(f(y))$ agrees with $\Eval_y \circ \pi$ on the unitary group and $h(\varepsilon)$-approximates $\Eval_y \circ \varphi$ on the unit ball. 
	
	By Lemma~\ref{L.FunctionfIsContinuous}, the function $f$ is continuous and $\Phi(a)(y)=a(f(y))$ is a $*$-homomorphism from $C(X)$ into $C(Y)$. For every $a\in C(X)$ we have $\|\Phi(a)-\varphi(a)\|\leq \sup_{y\in Y} |\Phi(a)(y)-\varphi(a)(y)|\leq h(\varepsilon)$, as required. This finishes the proof.
\end{proof}

\begin{lemma} \label{L.pi-to-rho-and-Phi.nonabelian} 
	There exists a function $h(\varepsilon)\to 0$ as $\varepsilon\to 0$ with the following property for sufficiently small $\varepsilon$. 
	Suppose that $A$ is a unital \cstar-algebra and $\pi\colon U(A)\to U(B)$ is a group homomorphism such that some $\varepsilon$-$*$-homomorphism $\varphi\colon A\to B$ satisfies the following.  
	\begin{enumerate}
		\item If $x\in A$ and $\|x\|\leq 1$ then $\|\varphi(x)\|\leq 1$,
		\item $\varphi[A_{\textrm{sa}}]\subseteq B_{\textrm{sa}}$,   $\varphi(0)=0$, and $\varphi(1)=\pi(1)=1$, 
		\item $\sup_{u\in U(A)} \|\varphi(u)-\pi(u)\|\leq \varepsilon+120\varepsilon^2$.   
	\end{enumerate} 
	Then there is  a $*$-homomorphism $\Phi\colon A\to B$ such that $\Phi(u)=\pi(u)$ for all $u\in U(A)$ and $\pi(\exp(ia))=\exp(i\Phi(a))$ for all $a\in A_{\textrm{sa}}$, and finally 
	\[
	\sup_{\|a\|\leq 1} \|\Phi(a)-\varphi(a)\|\leq h(\varepsilon). 
	\]
\end{lemma}

\begin{proof} By 
	Lemma~\ref{L.pi-to-rho-and-Phi},   for every abelian \cstar-subalgebra $D$ of $A$ there is  a $*$-homomorphism $\Phi_D\colon D\to B$  as required.  Since every  self-adjoint element  belongs to an abelian \cstar-subalgebra, we may define $\Phi \colon A_\textrm{sa}\to B_{\textrm{sa}}$ by $\Phi(a)=\Phi_D(a)$ where $D$ is an arbitrary abelian \cstar-subalgebra of $A$ such that $a\in D$. Note that $\Phi(a)$ does not depend on the choice of $D$ because $\exp(i\Phi_D(ta))=\pi(\exp(ita))$ for $t \in \mathbb R$ determines it uniquely. 

	As a consequence, $\pi$ is continuous. Indeed, if $\|u-v\|\leq \delta$ for unitaries $u$ and $v$, then $\|\pi(u)-\pi(v)\|=\|\pi(uv^*)-1\|=\|\pi(\exp(i a)) -1\|$ for some self-adjoint $a$ with $\|a\|=O(\delta)$. Since $\pi(\exp(i a))=\exp(i \Phi(a))$, we have $\|\pi(u)-\pi(v)\|=O(\delta)$ as well.

	We now prove that $\Phi \colon A_{\textrm{sa}} \to B_{\textrm{sa}}$ is linear. This is a consequence of the well-known Lie--Trotter product formula, or rather its special case for bounded operators.  It asserts that if $a$ and $b$ are skew-adjoint bounded linear operators on $H$, then  (see \cite[Exercise 15, p. 67]{Trotter})
	\begin{equation}
		\label{eq.Trotter}
		\exp(a+b)=\lim_{n\to \infty} (\exp(a/n)\exp(b/n))^n, 
	\end{equation}
	where the convergence is in the operator norm topology.

	For every self-adjoint operator $a$, we have that 
	$
	\pi(\exp(ia))=\exp(i\Phi(a)). 
	$
 	Fix such $a, b \in A_{\textrm{sa}}$. 
	By applying \eqref{eq.Trotter} in $A$, then 
	by continuity  of $\pi$ and applying \eqref{eq.Trotter} in $B$, we obtain the following. 
	\begin{align*}
		\exp(i\Phi(a)+i\Phi(b))&=\lim_{n\to \infty} (\exp(i\Phi(a)/n)\exp(i\Phi(b)/n))^n\\
		&=\lim_{n\to \infty} \pi((\exp(ia)/n)\exp(ib/n))^n)\\
		&=\pi(\exp(i(a+b)))\\
		&=\exp(i\Phi(a+b)). 
	\end{align*}
	If the norms of $\Phi(a)$, $\Phi(b)$, and $\Phi(a+b)$ are all small, then by taking logarithms we conclude that $\Phi(a)+\Phi(b)=\Phi(a+b)$. Since $\Phi(ta)=t\Phi(a)$ for all $a$ and scalars $t$ (this follows because $\Phi$ is a $*$-homomorphism on $\cst(a)$), we conclude that $\Phi \colon A_{\textrm{sa}} \to B_{\textrm{sa}}$ is $\mathbb R$-linear.

	We now extend $\Phi$ to a $*$-preserving complex linear map on $A$ by defining $\Phi(a+ib)=\Phi(a)+i\Phi(b)$ for self-adjoint $a$ and $b$. 

We claim that $\Phi(u)=\pi(u)$ for all unitaries $u\in A$.
Indeed, let $D=C^*(u)$. Then $D$ is abelian, and by construction
$\Phi$ agrees with $\Phi_D$ on $D_{\rm sa}$. Since both maps are complex
linear, they agree on all of $D$. Hence
$
\Phi(u)=\Phi_D(u)=\pi(u).
$

Since every element of $A$ can be presented as a linear combination of four unitaries, 
	$\Phi$ is multiplicative. It follows that $\Phi$ is a $*$-homomorphism. Similarly, for every $a\in A$ with $\|a\|\leq 1$ and since $\Phi(u)=\pi(u) \approx_{\varepsilon + 120 \varepsilon^2} \varphi(u)$ for all unitaries $u\in U(A)$, we have $\|\Phi(a)-\varphi(a)\|\leq 3(\varepsilon + 120 \varepsilon^2)$, as required.
\end{proof}

\begin{proof}[Proof of Theorem~\ref{T.Ulam2}]
	Fix $\varepsilon>0$ and $h(\varepsilon)$ as in Lemma~\ref{L.pi-to-rho-and-Phi}. 
	Fix a unital \cstar-algebra $A$, \cstar-algebra $B$,  an $\varepsilon$-$*$-homomorphism $\varphi\colon A \to B$, and a group homomorphism $\pi\colon U(A)\to U(B)$ such that $\sup_{u\in U(A)} \|\varphi(u)-\pi(u)\|\leq \varepsilon+120\varepsilon^2$. Assume the additional conditions ensured by Corollary \ref{C.UlamStableCorrected}.
	
	By Lemma~\ref{L.pi-to-rho-and-Phi.nonabelian}, 
	there exists a $*$-homomorphism $\Phi\colon A\to B$ such that $\Phi(u)=\pi(u)$ for all $u\in U(A)$ and 
	\(\sup_{\|a\|\le 1}\|\varphi(a)-\Phi(a)\|\le f(\varepsilon)\)
	as required. 
\end{proof}

An immediate application of these results is the following theorem, where by $\sfCOM$ we denote  the class of all abelian \cstar-algebras. 

\begin{theorem} \label{T.Abelian+} Each of the following pairs is Ulam stable. 
	\begin{enumerate}
		\item \label{U.Abelian.vN} $(\sfCOM,\sfvN)$. 
		\item \label{U.Abelian.Abelian}
		$(\sfCOM,\sfCOM)$. 
	\end{enumerate}
\end{theorem}

\begin{proof} \eqref{U.Abelian.vN} 
	We first prove the unital version of \eqref{U.Abelian.vN}, that $(\sfUC, \sfvN)$ is Ulam stable. 
	
	We will prove that $f_{\sfUC}=h$ as in Lemma~\ref{L.pi-to-rho-and-Phi} and Theorem~\ref{T.Ulam2} witnesses Ulam stability of abelian \cstar-algebras with respect to von Neumann algebras. 
	 Fix $\varepsilon>0$, an abelian unital separable \cstar-algebra $A$, von Neumann algebra $M$,  and an $\varepsilon$-$*$-homomorphism $\varphi\colon A \to M$. 
We claim that there exists a group homomorphism $\psi\colon U(A)\to U(M)$ such that
\[\sup_{\|a\|\le 1}\|\varphi(a)-\psi(a)\|\le \varepsilon+120\varepsilon^2.\]
Since $U(A)$ is amenable, the case when $M=B(H)$ is 
\cite[Theorem~2.1]{BMT}, also \cite{Kaz:epsilon}, but it is clear that the integration argument in the proof uses only the assumption that $M$ is a von Neumann algebra. 
By Theorem~\ref{T.Ulam2}, 
 there exists a $*$-homomorphism $\psi\colon A\to M$ such that
$\sup_{\|a\|\le 1}\|\varphi(a)-\psi(a)\|\le h(\varepsilon)$, as required. 

This proves the unital case, and 
it remains to prove that $(\sfCOM,\sfvN)$ is Ulam stable. 
In Theorem~\ref{T.Closure}\eqref{I.Ideals} below we will prove that if a class $\sfC$ is Ulam stable with respect to $\sfvN$, then so is the class of all two-sided, norm-closed ideals of \cstar-algebras in $\sfC$, and the full result  follows.

\eqref{U.Abelian.Abelian}
Suppose that $\varphi\colon C_0(X)\to C_0(Y)$ is an $\varepsilon$-$*$-homomorphism between unital abelian \cstar-algebras for  $\varepsilon>0$ that is sufficiently small as per \eqref{U.Abelian.vN}. Then  for every $y\in Y$  the $\varepsilon$-$*$-homomorphism $\Eval_y\circ \varphi$ can be $h(\varepsilon)$-approximated by a $*$-homomorphism $\Phi_y$. Let $f(y)\in X$ be such that $\Phi_y(a)=a(f(y))$ for all $a\in C_0(X)$. By Lemma~\ref{L.FunctionfIsContinuous}, the function $f$ is continuous and it determines a $*$-homomorphism $\Phi\colon C_0(X)\to C_0(Y)$ such that $\|\Phi(a)-\varphi(a)\|\leq h(\varepsilon)$ for all $a\in C_0(X)$. 
\end{proof}

\section{Simple closure properties}\label{S.SimpleClosure}

In this section we focus on Ulam stability with respect to von Neumann algebras,  which we call Ulam stability for the sake of brevity.

\begin{theorem}
\label{T.Closure}
The class of Ulam stable $C^*$-algebras is closed by taking:
\begin{enumerate}
\item \label{I.Ideals} norm-closed,  two-sided,  ideals,
\item \label{I.Quotients} quotients by norm-closed,  two-sided ideals,
\item \label{I.Extensions} extensions,
\item \label{I.Pullback} pullbacks where one map is surjective,
\item \label{thm:Johnson} inductive limits, and
\item \label{I.Tensor}tensor products with abelian \cstar-algebras.
\end{enumerate}
Moreover, if $\sfC$ is a class of \cstar-algebras such that the pair $(\sfC,\sfvN)$ is Ulam stable, then so is each of the classes $\sfC'$obtained from $\sfC$ by applying any of the operations listed above, with a possibly different witnessing function.
\end{theorem}

We are now going to prove these items one by one. The proof of Theorem~\ref{T.Closure}\eqref{I.Ideals} is a direct consequence of the following lemma.
In each of the proofs we fix a class $\sfC$ that is Ulam stable with respect to $\sfvN$. 

\begin{lemma}\label{L.extend.from.ideal}
Let $A$ be a \cstar-algebra, let $I$ be a closed two-sided ideal of $A$,
let $M$ be a von Neumann algebra, and let
$
\varphi\colon I\to M
$ 
be an $\varepsilon$-$*$-homomorphism.
Then there is a $3\varepsilon$-$*$-homomorphism
$
\widetilde\varphi\colon A\to M
$
such that
\[
\sup_{x\in I_1}\|\widetilde\varphi(x)-\varphi(x)\|
\leq 2\varepsilon.
\]
\end{lemma}

\begin{proof}
Let $(e_\lambda)_{\lambda\in\Lambda}$ be a quasi-central approximate unit
for $I$ in $A$. Thus
\[
e_\lambda\in I_+,\qquad \|e_\lambda\|\leq 1, \qquad
e_\lambda x\to x,\qquad xe_\lambda\to x
\qquad (x\in I),
\]
and
$
\|e_\lambda a-ae_\lambda\|\to 0
$ for all $a\in A$.

Choose a cofinal ultrafilter $\cU$ on the directed set $\Lambda$.
For $a\in A$, define
\[
\widetilde\varphi(a)
:=
\WOTlim_{\lambda\to\cU}\varphi(e_\lambda a),
\]
as the ultraweak ultrafilter limit. This is well-defined:
if $a\in A_1$, then $e_\lambda a\in I_1$, hence
$\|\varphi(e_\lambda a)\|\leq 1$, and bounded subsets of the
von Neumann algebra $M$ are ultraweakly relatively compact.

We first record a useful continuity claim. Let $x_\lambda,y_\lambda\in I_1$ be nets such that
$\|x_\lambda-y_\lambda\|\to 0$. Then
\[
\limsup_\lambda
\|\varphi(x_\lambda)-\varphi(y_\lambda)\|
\leq 2\varepsilon.
\]
Indeed, this is a direct consequence of Lemma \ref{L.approx.continuity} \eqref{1.approx}.

We next compare $\widetilde\varphi$ to $\varphi$ on $I$.
Let $x\in I_1$. Since $e_\lambda x\to x$ in norm, the claim gives
\[
\limsup_\lambda
\|\varphi(e_\lambda x)-\varphi(x)\|
\leq 2\varepsilon.
\]
Passing to the ultraweak limit and using ultraweak closedness of norm balls,
we obtain
$
\|\widetilde\varphi(x)-\varphi(x)\|
\leq 2\varepsilon.
$ and hence
\[
\sup_{x\in I_1}
\|\widetilde\varphi(x)-\varphi(x)\|
\leq 2\varepsilon.
\]

We verify that $\widetilde\varphi$ is a
$3\varepsilon$-$*$-homomorphism. Norm control, approximate additivity, approximate homogeneity, and approximate preservation of the involution are clear.
Approximate multiplicativity uses separate continuity of the multiplication on bounded sets: Let $a,b\in A_1$. We use only separate ultraweak continuity of multiplication:
\begin{align*}
\widetilde\varphi(a)\widetilde\varphi(b)
&=
\left(
\WOTlim_{\lambda\to\cU}\varphi(e_\lambda a)
\right)
\left(
\WOTlim_{\mu\to\cU}\varphi(e_\mu b)
\right) \\
&=
\WOTlim_{\lambda\to\cU}
\left(
\varphi(e_\lambda a)
\left(
\WOTlim_{\mu\to\cU}\varphi(e_\mu b)
\right)
\right) \\
&=
\WOTlim_{\lambda\to\cU}
\WOTlim_{\mu\to\cU}
\varphi(e_\lambda a)\varphi(e_\mu b).
\end{align*}
For fixed $\lambda$, approximate multiplicativity of $\varphi$ gives
$
\|\varphi(e_\lambda a)\varphi(e_\mu b)
-\varphi(e_\lambda a e_\mu b)\|
\leq \varepsilon.
$
Now $e_\lambda a\in I$, hence
$
(e_\lambda a)e_\mu\to e_\lambda a
$
in norm as $\mu\to\infty$. Therefore
$
e_\lambda a e_\mu b\to e_\lambda ab
$
in norm. By the claim,
\[
\limsup_\mu
\|\varphi(e_\lambda a e_\mu b)-\varphi(e_\lambda ab)\|
\leq 2\varepsilon.
\]
Hence
\[
\left\|
\WOTlim_{\mu\to\cU}
\varphi(e_\lambda a)\varphi(e_\mu b)
-\varphi(e_\lambda ab)
\right\|
\leq 3\varepsilon.
\]
Taking the outer ultraweak limit in $\lambda$ yields
$
\|\widetilde\varphi(a)\widetilde\varphi(b)
-\widetilde\varphi(ab)\|
\leq 3\varepsilon.
$
Combining the estimates above, $\widetilde\varphi$ is a
$3\varepsilon$-$*$-homomorphism, and
\[
\sup_{x\in I_1}
\|\widetilde\varphi(x)-\varphi(x)\|
\leq 2\varepsilon.
\]
This finishes the proof.
\end{proof}

\begin{proof}[Proof of Theorem~\ref{T.Closure}\eqref{I.Quotients}] Fix $\varepsilon>0$ such that $f_\sfC(\varepsilon)<1$. 
	If $\varphi\colon A/I\to D$ is an $\varepsilon$-$*$-homomorphism then so is its composition $\varphi\circ \pi_I$ with the quotient map from $A$ onto $A/I$. Suppose that  $\Phi\colon A\to D$ is a $*$-homomorphism that $f_{\sfC}$-approximates $\varphi\circ \pi_I$ on $A_1$. Thus, $\|\Phi(a)\| < |f_{\sfC}(\varepsilon)|<1$ for all contractions. Since every nonzero $*$-homomorphism has norm one, this implies that $I \subset \ker(\Phi)$ and $\Phi$ factors through $A/I$ and approximates $\varphi$, as required.   
\end{proof}

\begin{proof}[Proof of Theorem~\ref{T.Closure}\eqref{I.Extensions}] Fix $\varepsilon>0$ such that $f_\sfC(\varepsilon)<1$.
Suppose that $E$ is a \cstar-algebra with a two-sided norm closed ideal $I$ such that both $I$ and $C:=E/I$ belong to $\sfC$.
Let $\varphi\colon E\to M$ be an $\varepsilon$-$*$-homomorphism into a von Neumann algebra $M$. 

By Corollary~\ref{C.UlamStableCorrected} we may assume that $\varphi$ satisfies conditions \ref{Corrected.norm}--\ref{Corrected.log}. 
By the assumption, there is a $*$-homomorphism $\Phi_I\colon I\to M$ such that $\sup_{a\in I_1} \|\Phi_I(a)-\varphi(a)\|\leq f_{\sfC}(\varepsilon)$, where $f_{\sfC}$ witnesses Ulam stability  of $\sfC$ with respect to $\sfvN$.  

Fix an quasi-central approximate unit $(e_\lambda)_{\lambda\in \Lambda}$ of $I$ (\cite{Arv:Notes}, \cite[\S 1.9]{FarahCT}) and let $p$ be the supremum of  the directed set of positive contractions $\Phi_I(e_\lambda)$, for $\lambda\in \Lambda$. 
This is a projection that does not depend on the choice of $(e_\lambda)$. 
Fix $a\in E$. We will prove that 
$\|[\varphi(a),p]\|\leq 3\varepsilon +f_{\sfC}(\varepsilon)$. 
 
Since $(e_\lambda)$ is quasi-central, i.e., 
$\lim_\lambda \|[a,e_\lambda]\|=0$, we have $\limsup_\lambda \|[\varphi(a),\varphi(e_\lambda)]\|\leq 3\varepsilon$.
  Therefore, 
\begin{align*}
	\|[\varphi(a),p]\|&\leq \limsup_\lambda \|[\varphi(a),\Phi_I(e_\lambda)]\|\\
	&\leq \limsup_\lambda\|[\varphi(a),\varphi(e_\lambda)]\|+2\|\varphi(a)\|f_{\sfC}(\varepsilon)\\
	&\leq 3\varepsilon +2f_{\sfC}(\varepsilon), 
\end{align*}
as required. 
Therefore, each one of $\tilde\varphi_0,\tilde\varphi_1\colon E\to M$ defined by 
\begin{align*}
\tilde \varphi_0(a)&=p\varphi(a)p,\\
\tilde \varphi_1(a)&=(1-p)\varphi(a)(1-p)
\end{align*}
is an $\varepsilon'$-$*$-homomorphism for an $\varepsilon'$ that depends only on $\varepsilon$ and $f_{\sfC}(\varepsilon)$.
Moreover, by \ref{Corrected.leq} every positive contraction $a\in I$ satisfies 
\[
\varphi(e_\lambda^{1/2} a e_{\lambda}^{1/2})\leq \limsup_\lambda \varphi(e_\lambda)+3\varepsilon\leq \limsup_\lambda \Phi(e_\lambda)+2f_{\sfC}(\varepsilon)+3\varepsilon=p+2f_{\sfC}(\varepsilon)+3\varepsilon,  
\] 
and therefore $\|\tilde \varphi_1(a)\|\leq f_\sfC(\varepsilon)$  for all positive contractions $a\in I$. 
Thus, for $\varepsilon''$ depending only on $\varepsilon$ and $f_\sfC(\varepsilon)$,  $\tilde \varphi_1\circ s$ (where $s\colon C\to E$ is a norm-decreasing section) is an $\varepsilon''$-$*$-homomorphism from $C$ into $M$, and it can therefore be $f_{\sfC}(\varepsilon'')$-approximated by a $*$-homomorphism $\Phi_C\colon C\to M$.

We define $\Phi_0(a)= \WOTlim_{\lambda\to \cU} \Phi_I(e_\lambda a)$, where $\cU$ is a cofinal ultrafilter on $\Lambda$. Then $\Phi_0 \colon E \to M$ is a $*$-homomorphism by Lemma \ref{L.extend.from.ideal} for $\varepsilon=0$.
Fix $a\in E$. Then every $\lambda\in \Lambda$ satisfies 
\[
(\Phi_0(a)-\tilde\varphi_0(a))\Phi_I(e_\lambda)
\approx_ {f_\sfC(\varepsilon)} \Phi_I(ae_\lambda)-\tilde \varphi_0(a e_\lambda)\approx _{f_\sfC(\varepsilon)}0, 
\]
and therefore   $\|\Phi_0(a)-\tilde \varphi_0(a)\|\leq 2f_\sfC(\varepsilon)$.

Then, writing $\Phi_1=\Phi_C\circ \pi_I$ (where $\pi_I\colon E\to C$ is the quotient map) we have that 
\[
\Phi(a)=p \Phi_0(a) +(1-p)\Phi_1(a)
\]
is a $*$-homomorphism that approximates $\varphi$, as required. 
\end{proof}

\begin{proof}[Proof of Theorem~\ref{T.Closure}\eqref{I.Pullback}] Recall the definition of a pullback (\cite{pedersen1999pullback}). If $A,B,C$ are \cstar-algebras and $\alpha\colon A\to C$ and $\beta\colon B\to C$ are $*$-homomorphisms, then we define the pullback of this diagram by 
\[
A\oplus_C B =\{(a,b)\in A\oplus B: \alpha(a)=\beta(b)\}. 
\]
If $\alpha$ is surjective, then $A\oplus_C B$ is an extension of $B$ by $\ker(\alpha)$. Thus $A\oplus_C B$ is Ulam stable by Theorem~\ref{T.Closure}\eqref{I.Ideals}+\eqref{I.Extensions}. 
\end{proof}

\begin{proof}[Proof of Theorem~\ref{T.Closure}\eqref{thm:Johnson}] If $\sfC$ is a class of \cstar-algebras, we denote by $\sfAC$ the class of all inductive limits of inductive systems (not necessarily countable) of \cstar-algebras in $\sfC$. The fact that $\sfAC$ is Ulam stable if $\sfC$ is Ulam stable is a consequence of \cite{johnson1988approximately} stated explicitly and proved in \cite[Theorem~3.1]{MV}. 
\end{proof}

\begin{proof}[Proof of Theorem~\ref{T.Closure}\eqref{I.Tensor}] 
	Fix an abelian $A$, $C\in \sfC$, and an $\varepsilon$-$*$-homomorphism $\varphi\colon A\otimes  C\to M$. Without loss of generality, we may assume that $C$ is unital.
	In the following, by  $A^+$ we denote the unitization of a \cstar-algebra $A$.  
	By Lemma~\ref{L.extend.from.ideal} we may extend $\varphi$ to $A^+\otimes C$ and assume that $A$ is unital. 
	
	By Theorem~\ref{T.Abelian+} \eqref{U.Abelian.vN}, there is a $*$-homomorphism $\Phi_A\colon A\to M$ that $f_{\sfCOM}(\varepsilon)$-approximates the restriction of $\varphi$ to the unit ball of $A$. Since $M$ is a von Neumann algebra and the unitary group of $A$ is abelian, we can define $\varphi_1\colon A\otimes C\to M$ by  
	\[
	\varphi_1(x)=\int \Phi_A(u)\varphi(x)\Phi_A(u^{-1})\, d\mu(u)
	\]
	where $\mu$ is an invariant mean on the unitary group of $A$ (considered as a discrete group). We have, for every $x\in (A\otimes C)_1$ and $u\in U(A)$ (by Corollary~\ref{C.UlamStableCorrected}   \eqref{1.C.Ulam.Corrected}, 
	 we may assume that $\varphi$ sends the norm-unit ball into the norm-unit ball)
	\begin{align*}
		\|\varphi(x)-\Phi_A(u)\varphi(x)\Phi_A(u^{-1})\|
		&\approx_{2f_{\sfCOM}(\varepsilon)} 
		\|\varphi(x)-\varphi(u)\varphi(x)\varphi(u^{-1})\|\\
		&\approx_{2\varepsilon}
		\|\varphi(x)-\varphi(uxu^{-1})\|=0. 
	\end{align*}
	Since $\varphi_1$ is $2(f_{\sfCOM}(\varepsilon)+\varepsilon)$-approximated by an $\varepsilon$-$*$-homomorphism $\varphi$, it is a $\delta$-$*$-homomorphism, with 
	\[
	\delta=3(f_{\sfCOM}(\varepsilon)+\varepsilon)+\varepsilon. 
	\]
	We gained centrality at the cost of increasing the error, i.e.
	$$\Phi_A(a)\varphi_1(x) = \varphi_1(x)\Phi_A(a) \quad (x \in A \otimes C, \quad a \in A).$$
	
	Note that the range of $\varphi_1$ is included in the commutant
	$\Phi_A[A]'$, which is again a von Neumann algebra.
	By Ulam stability of $\sfC$ with respect to von Neumann algebras,
	the restriction of $\varphi_1$ to $1\otimes C$ can be
	$f_{\sfC}(\delta)$-approximated by a $*$-homomorphism
	$
		\Phi_C\colon C\to \Phi_A[A]'.
	$
	Since the ranges of $\Phi_A$ and $\Phi_C$ commute, we obtain a
	$*$-homomorphism
	$
		\Phi=\Phi_A\otimes \Phi_C\colon A\otimes C\to M.
	$
	We claim that $\Phi$ approximates $\varphi_1$ uniformly on the unit ball.

	Let $X$ be the spectrum of $A$, so that $A=C(X)$.
	Set
	\[
		N:=W^*\bigl(\varphi_1[A\otimes C],\,\Phi_A[A],\,\Phi_C[C]\bigr)
		\subseteq M.
	\]
	Since $\varphi_1[A\otimes C]\subseteq \Phi_A[A]'$ and
	$\Phi_C[C]\subseteq \Phi_A[A]'$, the algebra $\Phi_A[A]$ is contained
	in the centre of $N$. Hence $N$ is a $C(X)$-algebra. For $x\in X$, let
	$
		N_x:=N/C_0(X,x)N
	$
	and denote the quotient map by $q_x\colon N\to N_x$.
	For $z\in N$ we have
	\[
		\|z\|=\sup_{x\in X}\|q_x(z)\|.
	\]

	Note that$\Phi_A(a)\varphi_1(x) \approx_{2\delta} \varphi_1(ax)$ for all $x \in A \otimes C$ and $a \in A$.
	We first record a consequence of this approximate compatibility with the $C(X)$-module structure.
	If $d\in (A\otimes C)_1$ satisfies $d(x)=0$, then
	$
		\|q_x(\varphi_1(d))\|\leq 4\delta.
	$
	Indeed, given $\eta>0$, choose $a\in A_+$ with
	$0\leq a\leq 1$, $a(x)=0$, and $\|d-ad\|<\eta$.
	Using approximate additivity and approximate homogeneity of
	$\varphi_1$, we have
	$
		\|\varphi_1(d)-\varphi_1(ad)\|
		\leq \eta+2\delta.
	$
	On the other hand,
	$
		q_x(\varphi_1(ad))
		\approx_{2\delta}q_x(\Phi_A(a)\varphi_1(d))=0,
	$
	because $a(x)=0$. Letting $\eta\to 0$ gives the claim.

	Now fix $b\in (A\otimes C)_1$ and $x\in X$, and put
	$
		c_x:=1\otimes b(x)\in A\otimes C.
	$
	Then $e_x:=\frac12(b-c_x)$ belongs to $(A\otimes C)_1$ and satisfies
	$e_x(x)=0$. By the preceding claim,
	$
		\|q_x(\varphi_1(e_x))\|\leq 4\delta.
	$
	Using approximate additivity and approximate homogeneity once more,
	\begin{align*}
		\|q_x(\varphi_1(b)-\varphi_1(c_x))\|
		&\leq \|q_x(\varphi_1(b-c_x))\|+2\delta\\
		&=\|q_x(\varphi_1(2e_x))\|+2\delta\\
		&\leq 2\|q_x(\varphi_1(e_x))\|+3\delta\\
		&\leq 10\delta.
	\end{align*}
	Since $\Phi$ is a $C(X)$-homomorphism,
	$
		q_x(\Phi(b))=q_x(\Phi(c_x))
		=q_x(\Phi_C(b(x))).
	$
	Moreover, by the choice of $\Phi_C$,
	$
		\|\varphi_1(c_x)-\Phi_C(b(x))\|
		\leq f_{\sfC}(\delta).
	$
	Therefore,
	\begin{align*}
		\|q_x(\varphi_1(b)-\Phi(b))\|
		&\leq
		\|q_x(\varphi_1(b)-\varphi_1(c_x))\|
		+\|q_x(\varphi_1(c_x)-\Phi_C(b(x)))\|\\
		&\leq 10\delta+f_{\sfC}(\delta).
	\end{align*}
	Taking the supremum over $x\in X$, we obtain
	\[
		\|\varphi_1(b)-\Phi(b)\|
		\leq 10\delta+f_{\sfC}(\delta)
		\qquad (b\in (A\otimes C)_1).
	\]
	Since $\varphi_1$ uniformly approximates $\varphi$ on the unit ball
	by a quantity tending to $0$ with $\varepsilon$, the same is true of
	$\Phi$. This finishes the proof.
\end{proof}

\section{Ulam stability for Elliott classifiable \cstar-algebras}
\label{S.Extensions}

In this section we prove that the class that includes all Elliott classifiable \cstar-algebras and many nonclassifiable ones, such as all Villadsen algebras of the first type,   is Ulam stable.

\subsection{Homogeneous \cstar-algebras and their limits}

\begin{corollary}
	\label{T.AH}  The class 
	\[
	\sfC=\{A\otimes F: A\text{ abelian, $F$ finite-dimensional}\}
	\]
	is Ulam stable with respect to von Neumann algebras and so is 
	the class $\sfAC$ of all inductive limits of \cstar-algebras in $\sfC$.  
\end{corollary}
\begin{proof}
	For $\sfC$ this is a consequence of  the fact that the class of finite-dimensional \cstar-algebras is Ulam stable with respect to von Neumann algebras (Theorem~\ref{thm:MV_AF})  and the fact that tensoring with abelian \cstar-algebras preserves Ulam stability with respect to von Neumann algebras (Theorem~\ref{T.Closure}\eqref{I.Tensor}). 
	Ulam stability of $\sfAC$ follows from that of $\sfC$ and  the closure under inductive limits (Theorem~\ref{T.Closure}\eqref{thm:Johnson}). 
\end{proof}

The class of \cstar-algebras in Corollary~\ref{T.AH} includes all Villadsen algebras of the first type (see \cite[\S 3.1]{toms2009Elliott}). 

Now, we show that we do obtain uniform Ulam stability for all homogeneous \cstar-algebras when the base space has finite covering dimension.

Recall that for $n\in \bbN$ a \cstar-algebra is called \emph{$n$-homogeneous} if each one of its irreducible representations is of dimension $n$ (\cite[IV.1.4.1]{Black:Operator}). The base space of an $n$-homogeneous \cstar-algebra $A$, denoted $\hat A$,  is the space of its irreducible representations. In the following we deal only with compact metric spaces $X$ and $\dim(X)$ denotes the covering dimension of $X$. 

\begin{theorem}\label{T.Homogeneous}
	For every $m\in \bbN$, the class $\sfH_m$ of all \cstar-algebras that are $n$-homogeneous for some $n$ and whose base space is a compact metric space of covering dimension $\leq m$ is Ulam stable with respect to von Neumann algebras.
\end{theorem}
Later on, we will prove that the conclusion of Theorem~\ref{T.Homogeneous} holds without the assumption of separability (Corollary~\ref{C.Nonseparable}). 

The theorem will be proved after two lemmas. 

\begin{lemma}\label{L.Dimension}
	Suppose that $X$ is a compact metric space of dimension $m$ and $\cU$ is an open cover of $X$. Then there is a closed subspace $X_1$ of $X$ of dimension $\leq m-1$ such that each component of $X\setminus X_1$ is covered by a single element of $\cU$. 
\end{lemma}

\begin{proof}
	By compactness we may assume $\cU$ is finite and enumerate it as $U_j$, for~$j<k$. Fix an open $V_j\subseteq U_j$, for $j<k$, such that $\bigcup_{j<k} V_j=X$ and $F_j=\overline {V_j}$ is included in $U_j$ for all~$j$. By the First Separation Theorem  from dimension theory of separable metric spaces (\cite[Theorem~1.5.12]{engelking1978dimension}) there is a closed $G_j\subseteq X$ such that no component of $X\setminus G_j$ intersects both  $F_j$ and $X\setminus U_j$ nontrivially and $\dim(G_j)\leq m-1$ (i.e., $G_j$ is a \emph{partition between} $F_j$ and $X\setminus U_j$). 
	Then $X_1=\bigcup_{j<k} G_j$ has dimension $\leq \max_{j<k} \dim(G_j)\leq m-1$. 
	
	To prove $X_1$ is as required, let $C$ be a component of $X\setminus X_1$. Since $\bigcup_{j<k} F_j=X$, $C$ intersects some $F_j$ nontrivially, and by construction this implies $C\subseteq U_j$.  	
\end{proof}

\begin{lemma} \label{L.DimensionInduction} Suppose that $A$ is $n$-homogeneous with compact metric base space $X$ of finite dimension $m$. Then $A$ has a two-sided, norm-closed  ideal $I$ such that $I\cong C_0(Y,M_n(\bbC))$ for some open $Y\subseteq X$ and $A/I$ is $n$-homogeneous with compact metric base space of dimension $\leq m-1$. 
\end{lemma}

\begin{proof}
	Recall that if $A$ is $n$-homogeneous with compact Hausdorff base space $X$, then~$A$ is isomorphic to a \cstar-subalgebra of $C(X,M_{n'}(\bbC))$ for a sufficiently large multiple $n'$ of $n$, see \cite[Proposition 2.9]{phillips2007recursive}. While the bundle over $X$ associated with $A$ may not be trivial, it is locally trivial. 
	By \cite[IV.1.7.23]{Black:Operator},  every 
	$x\in X$ has an open neighbourhood $U$ such that 
	\[
	I(U,A)=\{a\in A: \supp(a)\subseteq U\}
	\]
	is isomorphic to $C_0(U,M_n(\bbC))$. 
	
	By compactness, there is a finite open cover $\cU$ of $X$ such that $I(U,A)\cong C_0(U,M_n(\bbC))$ for every $U\in \cU$. 
	Lemma~\ref{L.Dimension} implies that there is $X_1\subseteq X$ of dimension $\leq m-1$ such that every component of $X\setminus X_1$ is covered by a single element of $\cU$. 
	Therefore, with $U=X\setminus X_1$, we have $I(U,A)\cong C_0(U,M_n(\bbC))$. The quotient $A/I(U,A)$ is $n$-homogeneous and its base space is $X_1$, as required. 
\end{proof}
	
\begin{proof}
	[Proof of Theorem~\ref{T.Homogeneous}]
We can now prove the theorem by induction on $m$. Suppose that $m=0$. 
If $A$ is $n$-homogeneous and its base space $X$ is 0-dimensional, then there is a finite  open cover $\cU$ of $X$ such that every $U\in \cU$ satisfies $I(U,A)\cong C_0(U,M_n(\bbC))$. Since $X$ is 0-dimensional, we may assume that the sets in $\cU$ are compact and disjoint, and therefore $A\cong \bigoplus_{U\in \cU} C(U,M_n(\bbC))\cong C(X,M_n(\bbC))$.
The class $\sfH_0$ of \cstar-algebras of this form is therefore Ulam stable by Corollary~\ref{T.AH}. 

Fix $m\geq 0$ and suppose that $\sfH_m$ 
 is Ulam stable with respect to von Neumann algebras. By Lemma~\ref{L.DimensionInduction}, every $A\in \sfH_{m+1}$ 
 has a two-sided, norm-closed  ideal isomorphic to 
 $C_0(Y,M_n(\bbC))$ for some locally compact metric $Y$ and $A/I$ is $m$-homogeneous with compact metric base space of dimension $\leq m$, hence it belongs to $\sfH_m$. By Corollary~\ref{T.AH} and the inductive hypothesis, the classes of such $I$ and $A/I$ is are Ulam stable with respect to von Neumann algebras, and the conclusion follows by Theorem~\ref{T.Closure}\eqref{I.Extensions}. 
\end{proof}

\begin{corollary}
The class ${\sf AH}_m$ is Ulam stable with respect to von Neumann algebras for every $m\in \bbN$.	
\end{corollary}
\begin{proof} This is a direct consequence of Theorem~\ref{T.Homogeneous} and the closure under inductive limits (Theorem~\ref{T.Closure}\eqref{thm:Johnson}).
\end{proof}

\subsection{Recursive subhomogeneous algebras}

We now turn the study of recursively defined classes of \cstar-algebras and their inductive limits. This provides a class rich enough to include all Elliott classifiable \cstar-algebras.

\begin{corollary}\label{T.pullback}
	Suppose that $\sfB$ is Ulam stable with respect to von Neumann algebras. Then the class $\sfB'$ of all pullbacks of the form 
	\[
	B\oplus _{C(Y,F)} C(X,F)
	\]
	where $B$ and $F$ are in $\sfB$, $X$ is compact metrizable, $Y\subseteq X$ is closed, $\beta\colon B\to C(Y,F)$ is a $*$-homomorphism, and $\gamma\colon C(X,F)\to C(Y,F)$ is the restriction map, is Ulam stable with respect to von Neumann algebras. 
\end{corollary}
\begin{proof}
	By Corollary \ref{T.AH} and Theorem~\ref{T.Closure}\eqref{I.Pullback}, it suffices to notice that the map $\gamma$ is surjective.
\end{proof}

We now iterate Theorem~\ref{T.pullback}. The class of recursive subhomogeneous \cstar-algebras was introduced in \cite{phillips2007recursive} and the class of NCCW complexes was introduced in \cite[\S 11]{pedersen1999pullback}. Let us say that the \emph{length} of a NCCW complex $A_n$ as defined in \cite[Definition~11.2]{pedersen1999pullback} is the number of iterations of the pullback construction used to construct it. 

\begin{corollary} \label{C.Recursive} For every $n\geq 1$, 
	the class $\mathcal R_n$ of recursive subhomogeneous \cstar-algebras of length $\leq n$ and the class $\sfAR_n$ of inductive limits of \cstar-algebras in $\cR_n$  are Ulam stable with respect to von Neumann algebras.  
\end{corollary}

\begin{proof}By \cite[Definition~1.2]{phillips2007recursive}, these \cstar-algebras are obtained from \cstar-algebras of the form $C_0(X,M_n(\bbC))$ for $n\in \bbN$ and locally compact Hausdorff $X$ by an $n$-fold application of the pullback construction. Therefore the conclusion follows by  Theorem~\ref{T.pullback} and induction. 
NCCW complexes of length $n$ are a special case of recursive subhomogeneous algebras of length $n$ (\cite[Example~1.5]{phillips2007recursive}). 
\end{proof}

Thus $\mathcal R_1$ already contains the generalized dimension drop algebras, while the higher classes are obtained by iterating the same pullback construction. 

Corollary~\ref{C.Recursive} applies to every ASH (approximately subhomogeneous) algebra that admits an inductive limit decomposition by recursive pullback stages of uniformly bounded length.  In particular, it applies to all finite-stage recursive subhomogeneous models built from homogeneous blocks by a bounded number of pullback steps.

By \cite[Theorem 6.2 (iii)]{tikuisis-white-winter}, every separable simple unital nuclear stably finite $\mathcal Z$-stable \cstar-algebra satisfying the UCT is an inductive limit of subhomogeneous \cstar-algebras of topological dimension at most 2.  More precisely, by \cite[1.9--1.11]{winter:decomposition-length} all building blocks in the inductive limit decomposition belong to the class $\mathcal R_2$ described above. Therefore Corollary~\ref{C.Recursive} applies to each building block, and since the control function is uniform along the inductive limit, we obtain the following. 

\begin{theorem} \label{T.Classifiable}
	The class of all separable, simple, unital, nuclear, stably finite, $\mathcal Z$-stable \cstar-algebras satisfying the UCT, is Ulam stable with respect to von Neumann algebras. \qed 
\end{theorem}


It would thus be of interest to understand which nuclear \cstar-algebras beyond the above admit inductive limit decompositions by recursive pullback stages of uniformly bounded complexity.
We do not know whether the class $\bigcup_n \cR_n$ is Ulam stable with respect to von Neumann algebras.

\subsection{Reflection and nonseparable \cstar-algebras}\label{S.Reflection}

In this section we show how standard methods  imply that all of our positive results on Ulam stability hold without the assumption of separability. This should be contrasted with the Choi--Christensen result showing that Kadison--Kastler stability fails for separable, nuclear \cstar-algebras (\cite{choi1983completely}, see also \cite[\S 14.4]{FarahCT}). 

The following is \cite[Definition~7.3.1]{FarahCT}. 

\begin{definition}
Given a not necessarily separable \cstar-algebra $A$, by $\Sep(A)$ we denote the set of all separable \cstar-subalgebras of $A$. A subset $\cC$ of $\Sep(A)$ is said to be a \emph{club} if (i) it is cofinal in $\Sep(A)$, meaning that for every $B\in \Sep(A)$ some $C\in \cC$ satisfies $B\subseteq C$ and (ii) it is closed, meaning that if $B_n\in \cC$, for $n\in \bbN$, is an increasing sequence then $\overline {\bigcup_n B_n}\in \cC$.\footnote{`Club' stands for `closed, unbounded', although  `closed, cofinal' would be more accurate.}
A property $P$ of \cstar-algebras is said to \emph{reflect to separable substructures} if for every \cstar-algebra $C$ with property $P$, the set of all separable \cstar-subalgebras of $C$ with property $P$ includes a club. 
\end{definition}

\begin{proposition} \label{P.Reflection} Suppose that $P$ is a property of \cstar-algebras such that the class of all separable \cstar-algebras with property $P$ is Ulam stable with respect to $\sfvN$. Then the class of all \cstar-algebras with property $P$ is Uam stable with respect to $\sfvN$. \qed 
\end{proposition}
\begin{proof}
	If $\cC\subseteq \Sep(A)$ is a club, then $A$ can be presented as an inductive limit of \cstar-algebras in $\cC$. Therefore, the conclusion follows by  Theorem~\ref{T.Closure} \eqref{thm:Johnson}. 
\end{proof}

\begin{corollary}\label{C.AxiomatizableReflects}
	If $P$ is a property of \cstar-algebras that is axiomatizable in continuous logic, then the class of all \cstar-algebras with property $P$ is Ulam stable if and only if the class all separable \cstar-algebras with property $P$ is Ulam stable. 
\end{corollary}

\begin{proof}
	Only the converse implication requires a proof. By the L\"owenheim--Skolem theorem, $P$ reflects to separable \cstar-subalgebras and the conclusion follows by Proposition~\ref{P.Reflection}. 
\end{proof}

\begin{corollary}\label{C.Nonseparable}
	For every $m\in \bbN$, the class of all \cstar-algebras that are $n$-homogeneous for some $n$ and whose base space is a compact Hausdorff space of covering dimension $\leq m$ is Ulam stable with respect to von Neumann algebras. 
\end{corollary}

\begin{proof}
	The conclusion will follow from Corollary~\ref{C.AxiomatizableReflects},  and 
	Theorem~\ref{T.Homogeneous}, and Lemma~\ref{L.HomogeneousReflects} below. 
\end{proof}

\begin{lemma}
	\label{L.HomogeneousReflects}
The property of being $n$-homogeneous reflects to separable \cstar-subalgebras. 
\end{lemma}

\begin{proof} Fix an $n$-homogeneous \cstar-algebra $A$ whose base space is a compact Hausdorff space of covering dimension $\leq m$. Since being $n$-subhomogeneous is axiomatizable by \cite[Theorem~2.5.1]{Muenster}, it reflects and  $\Sep(A)$ includes a club $\cC$ that consists of $n$-subhomogeneous \cstar-algebras. We need to prove that there is a club $\cC'\subseteq \cC$ that consists of $n$-homogeneous \cstar-algebras. Assume otherwise. Then  the set 
	\[
	\cB=\{B\in \cC: B\text{ is not $n$-homogeneous}\}
	\] 
	intersects every club in $\Sep(A)$ (i.e., this set is \emph{stationary}), in particular it is cofinal in $\Sep(A)$.  Thus every $B\in \cB$ has an irreducible representation on a $k$-dimensional Hilbert space for $k<n$. The union of finitely (even countably) many nonstationary sets is nonstationary, for some $k<n$ the set $\cC_k$ of all $B\in \cC$ such that $B$ has an irreducible representation $\pi_B\colon B\to M_k(\bbC)$ is stationary. Fix a cofinal ultrafilter $\cU$ on $\cC_k$. For $a\in A$ let 
	\[
	\pi(a)=\lim_{B\to \cU} \pi_B(a). 
	\]
Since for every $a\in A$, $\pi_B(a)$ is defined for $\cU$ many $B$ and belongs to the norm-compact ball of $M_k(\bbC)$, this function is well-defined.  Therefore $\pi$ is a unital *-homomorphism of $A$ into $M_k(\bbC)$, contradicting the assumption that $A$ is $n$-homogeneous.  
\end{proof}

\section{Examples and non-examples}

\subsection{Johnson's counterexamples to Ulam stability}
\label{S.Johnson}

We are grateful to Stuart White for bringing the highly relevant  \cite[Theorem~3 and Corollary~4]{johnson1982counterexample}, restated in the following proposition for reader's convenience,  to our attention.

\begin{proposition}\label{P.Counterexample} 
Let $A=C([0,1],\cK(H))$, where $\cK(H)$ is the algebra of compact operators on a separable Hilbert space $H$. 
	\begin{enumerate}
		\item \label{1.Counterexample} The pair $(\cK(H), \{A\})$  is not Ulam stable. 
		\item \label{3.Counterexample} For every $\varepsilon>0$ there is an $\varepsilon$-$*$-isomorphism $\varphi_\varepsilon\colon A\to A$ such that no $*$-ho\-mo\-mor\-phism $\Phi\colon A\to A$ satisfies $\sup_{\|a\|\leq 1} \|\Phi(a)-\varphi_\varepsilon(a)\|<1/80$. 
		\item 	\label{2.Counterexample} The pair $(\{c_0\}, \{C([0,1],\cK(H)\})$  is not Ulam stable. 
	\end{enumerate}
	\end{proposition}

\begin{proof} This is an immediate consequence of Johnson's results, but  since we found the notation of this paper somewhat difficult to penetrate, we first provide a brief guide to the results of \cite{johnson1982counterexample}. All references in this proof  are to this paper. 
		In Lemma~2 it was proved that for each $\eta>0$ there is  a self-adjoint contraction $a_\eta\in \prod_{x\in [0,1]} B(H)$ (where $[0,1]$ is taken with the discrete topology, as in the definition of $\frD$ below) with remarkable approximation properties. 
	Let $\frD=\prod_{x\in [0,1]} \cK(H)$, and consider $A=C([0,1],\cK(H))$ as a \cstar-subalgebra of $\frD$.

For $\varepsilon>0$ fix $\eta=\eta(\varepsilon)>0$ such that \begin{equation}\label{eq.eta}
\eta<\min\{1/2,(1/2)\varepsilon\exp(-3\pi/4)\}. 
\end{equation}
Let $u_\eta=\exp(i\pi a_\eta/8)$. Then $\alpha_\eta=\Ad u_\eta$ is an automorphism of $\frD$ and $d(A,\alpha_\eta [A])<\varepsilon$, but there is no $*$-homomorphism $\Phi\colon A\to \alpha_\eta[A]$ such that $\|\Phi(a)-a\| \leq (1/70)\|a\|$ for all $a\in A$. 
Even more is true. 
Let $j\colon \cK(H)\to A$ be the diagonal map, 
\[
j(c)=1_{C([0,1])}\otimes c. 
\]
Then Theorem 3 asserts that  there is no $*$-homomorphism from $C=(\alpha_\eta \circ j)[\cK(H)]$ into $A$ such that 
	$\|\Phi(c)-c\| \leq (1/70)\|c\|$ for all $c\in C$.  

	We  are now in a position to prove  \eqref{1.Counterexample} and \eqref{3.Counterexample}.   We will prove that 
	for every $\varepsilon>0$ there is a $\varepsilon$-$*$-isomorphism $\varphi_\varepsilon\colon A\to A$ such that $\Phi$ cannot be $\frac 1{80}$-approximated by a $*$-homomorphism. Fix $\varepsilon<\frac 1{560}$ and let $\eta$ be as in \eqref{eq.eta}.   With $j$, $u_\eta$, and $\alpha_\eta=\Ad u_\eta$  as before, for $b\in b\in (\alpha_\eta [A])_1$ let $\varphi_\varepsilon
	(b)$ be any $a\in A$ such that $\|(\alpha_\eta )(b)-a\|<\varepsilon\|b\|$. If $b\neq 0$ and $\|b\|\neq 1$ let $\varphi_\varepsilon(b)=\|b\|\varphi(b/\|b\|)$ and let $\varphi_\varepsilon)0)=0$. Then $\varphi_\varepsilon
	$ is a $3\varepsilon$-$*$-homomorphism. It will suffice to prove that, with $C=(\alpha_\eta \circ j)[\cK(H)]$,   no $*$-homomorphism $\Phi\colon C\to A$ satisfies $\|\Phi(a)-\varphi_\varepsilon(a)\|\leq 1/80$ for all $a\in C_1$ (the unit ball). Assume otherwise. Then  every nonzero $a\in C$ satisfies 
	\begin{align*}
		\|\Phi(a)-\alpha_\eta(a)\|&=\|a\|\,\|\Phi(a/\|a\|)-\alpha_\eta(a/\|a\|)\|\\
		&<\|a\|\, (\|\Phi(a/\|a\|)-\varphi_\varepsilon(a/\|a\|)\|+\delta)\\
		&< \|a\|(1/80+\delta). 
	\end{align*}  
	Since $\delta<\frac 1{70}-\frac 1{80}$, this is absurd, concluding the proof of \eqref{1.Counterexample}. 
	
	\eqref{2.Counterexample}
	    In Corollary~4 it  was proved that, with the notation as in \eqref{1.Counterexample}   and $c_0$ identified with the 
	   diagonal operators in $\cK(H)$,  no $*$-homomorphism 
	   $\Phi\colon (\alpha_\eta \circ j)[c_0]\to A$ satisfies $\|\Phi(a)-(\alpha_\eta \circ j)(a)\|\leq \frac 1{1000}\|a\|$ for all $a\in (\alpha_\eta \circ j)[c_0]$.

The fact that 	there is $\delta>0$ such that for every $\varepsilon>0$ the restriction of $\varphi_\varepsilon$ as defined in the proof of \eqref{1.Counterexample}  cannot be $\delta$-approximated by a $*$-homomorphism follows from this by computation analogous to that in the proof of \eqref{1.Counterexample}.  
\end{proof}

We can now prove 
Theorem~\ref{T.Nontrivial}. 

\begin{theorem} \label{T.Nontrivial+} Let $B$ denote $C([0,1],\cK(H))^+$. Then $B_\infty$ has an automorphism that is not algebraically trivial. \qed 
\end{theorem}
\begin{proof}
	Since every $\varepsilon$-isomorphism between non-unital \cstar-algebras can be extended  to an $\varepsilon$-isomorphism between their unitizations, by Theorem~\ref{P.Counterexample} \eqref{3.Counterexample} for every $n\geq 1$ we can fix a  $1/n$-isomorphism $\varphi_n\colon B\to B$ that cannot be $1/80$ approximated by a $*$-homomorphism. Then $(b_n)\mapsto (\varphi_n(b_n))$ lifts an automorphism $\Psi$ of $B_\infty$.  If $\Psi$ had an algebraically trivial lifting $\psi$, then for large enough $n$ the restriction of $\psi$ to the $n$-th copy of $B$ in $\ell_\infty(B)$ would be an automorphism, contradicting the choice of $\varphi_n$. 
\end{proof}

\subsection{Ulam stability for $\BH$}\label{S.B(H)}

Let $H$ be an infinite-dimensional separable Hilbert space.
We write $\BH$ for the algebra of bounded operators, $ K(H)\subseteq \BH$ for the ideal of compact operators, and $\QH=\BH/ K(H)$ for the Calkin algebra.
Let $\pi\colon \BH\to \QH$ denote the quotient map.

\medskip

The aim of this note is to show the following theorem

\begin{theorem}
There exists a function $f(\varepsilon)\to 0$ as $\varepsilon\to 0$ with the following property: Let $H$ be a separable Hilbert space and $\varphi\colon B(H) \to B(H)$ is an $\varepsilon$-$*$-homomorphism, then there exists a $*$-homomorphism $\psi\colon B(H)\to B(H)$ such that
\[\sup_{\|a\|\le 1}\|\varphi(a)-\psi(a)\|\le f(\varepsilon).\]
\end{theorem}

First of all, the  theorem of McKenney--Vignati (Theorem \ref{thm:MV_AF}\eqref{I.MV_fin}) allows us to replace any given $\varepsilon$-$*$-homomorphisms $\varphi \colon B(H) \to B(H)$ by a uniformly norm-close map $\psi$ such that $\psi|_{K(H)}$ is a $*$-homomorphism. Hence, without loss of generality, we assume that $\varphi|_{K(H)}$ is a $*$-homomorphism already. Now, fix an orthonormal basis $(e_k)_{k\ge1}$ and 
let $p_n$ be the projection onto $\mathrm{span}\{e_1,\dots,e_n\}$. Let
\[
\psi_n:\BH\to M_n(\bbC),\qquad \psi_n(x)=p_nxp_n,
\]
where we identify $p_n\BH p_n$ with $M_n(\bbC)$. We set $\varphi_n := \varphi|_{M_n} \circ \psi_n \colon \BH \to \BH$. Fix a free ultrafilter $\mathcal{U}$ on $\mathbb{N}$ and define a map $\Phi:\BH\to\BH$ by taking the weak ultralimit by the condition
\[
\langle \Phi(x)\xi,\eta\rangle := \lim_{n\to\mathcal{U}} \langle \varphi_n(x)\xi,\eta\rangle
\qquad (x\in\BH,\ \xi,\eta\in H).
\]
The limit exists for each matrix coefficient and defines a bounded operator, i.e., the unit ball in the norm is weakly compact.

\begin{proposition}
The map $\Phi: \BH_{\le 1}\to \BH_{\le 1}$ is linear, completely positive and contractive.
Moreover, for each compact operator $k\in K(H)$ one has $\Phi(k)=\varphi(k)$.
\end{proposition}

\begin{proof}
Linearity and complete positivity are preserved under pointwise ultraweak limits.
Contractivity holds because each $\varphi_n$ is contractive and the unit ball is ultraweakly closed. Since $\psi_n(k) = k$ eventually for a norm dense subset of $K(H)$, the rest of the claim follows.
\end{proof}

The standard multiplicative domain lemma goes back to Choi and is well-covered in operator algebra texts, see \cite[\S3]{Paulsen} or \cite{Choi74}. Let us recall it here for reader's convenience:

\begin{lemma}\label{lem:md}
Let $A,B$ be $C^*$-algebras and $\Psi\colon A\to B$ be unital completely positive.
Define the multiplicative domain
\[\mathrm{md}(\Psi)=\{a\in A: \Psi(a^*a)=\Psi(a)^*\Psi(a)\ \text{and}\ \Psi(aa^*)=\Psi(a)\Psi(a)^*\}.\]
Then $\mathrm{md}(\Psi)$ is a $C^*$-subalgebra of $A$, and for every $a\in\mathrm{md}(\Psi)$ and $x\in A$ one has
\[\Psi(ax)=\Psi(a)\Psi(x),\qquad \Psi(xa)=\Psi(x)\Psi(a).\]
In particular, if $\Psi$ is a $*$-homomorphism on a $C^*$-subalgebra $D\subseteq A$, then $D\subseteq \mathrm{md}(\Psi)$.
\end{lemma}

\begin{proposition}\label{prop:bimodule}
Then $ K(H)\subseteq \mathrm{md}(\Phi)$ and hence $\Phi$ is a $ K(H)$-bimodule map:
\[\Phi(k_1 x k_2)=\Phi(k_1)\,\Phi(x)\,\Phi(k_2)\qquad(k_1,k_2\in K(H),\ x\in\BH).\]
Moreover, letting $p:=\Phi(1)$, one has for all $x\in B(H)$:
\[\Phi(x)=p\Phi(x)=\Phi(x)p.
\]
For $x,y \in B(H)$, we also have $\Phi(xy)p=\Phi(x)\Phi(y)p$ for all $x,y \in B(H)$.
\end{proposition}

\begin{proof} Note that $p = \sup_n \varphi(p_n)$ and hence $\Phi(x)=p\Phi(x)p$ for all $x \in B(H)$. Consider the unital completely positive map $\Phi \colon B(H) \to p B(H)p$.
Since $\Phi|_{ K(H)}$ is a $*$-homomorphism, Lemma~\ref{lem:md} gives $ K(H)\subseteq \mathrm{md}(\varphi)$ and therefore bimodularity. For $k$ arbitrary compact, we compute $$\Phi(xy)\Phi(k)=\Phi(xyk) = \Phi(x)\Phi(yk)=\Phi(x)\Phi(y)\Phi(k).$$ Since $\sup_n\Phi(p_n) = p$, this implies the claim.
\end{proof}

Thus, $\Phi$ is multiplicative on the image of $p$ and it remains to show that the complementary part of $\varphi$ is uniformly small in the operator norm. We define
\[\varphi^{\perp}(x):=(1-p)\varphi(x)(1-p),\qquad x\in\BH.\] Note that for all compact operators $k \in K(H)$, we have
$$\varphi^{\perp}(k) = (1-p)\Phi(k)(1-p) =0.$$ Thus, picking a section $s \colon \QH_{\le 1} \to B(H)_{\le 1}$, we may change $\varphi^{\perp}$ by a uniformly small amount in operator norm and assume that it is well-defined on $\QH_{\le 1}$. We claim that the induced map $\varphi^{\perp} \colon Q(H)_{\le 1} \to B(H)_{\le 1}$ is almost multiplicative.

Indeed, for $x,y \in Q(H)_{\le 1}$ we have
\begin{eqnarray*}
\varphi^{\perp}(xy) - \varphi^{\perp}(x) \varphi^{\perp}(y) 
&=_{\hspace{0.13cm}}& (1-p)(\varphi(xy)-\varphi(x)\varphi(y))(1-p) - (1-p)\varphi(x)p \varphi(y) (1-p) \\
&=_{\varepsilon}& - (1-p)\varphi(x)p \varphi(y) (1-p)
\end{eqnarray*}

However, we have
\begin{eqnarray*} p \varphi(x)(1-p) 
&=_{\hspace{0.13cm}}&\WOTlim_{n \to \mathcal U} \varphi(p_n) \varphi(x) (1-p) \\
&=_{\varepsilon}&\WOTlim_{n \to \mathcal U} \varphi(p_n x) (1-p) \\
&=_{\varepsilon}&\WOTlim_{n \to \mathcal U} \Phi(p_n x)(1-p) \\
&=_{\hspace{0.13cm}}& 0.
\end{eqnarray*}

Thus, we conclude that that $\varphi^{\perp}$ is $3 \varepsilon$-multiplicative on $\QH_{\le 1}$. Approximate additivity and $*$-preservation are shown in a similar way.

Our aim is to prove that $\varphi^{\perp}(x)$ is uniformly small in the operator norm in $B(H)_{\le 1}$ for all $x \in Q(H)_{\le 1}$. We are going to use a quantitative version of the proof that $\QH$ has no $*$-representation on a separable Hilbert space. This argument is based on the fact that $\QH$ contains an uncountable family of pairwise orthogonal nonzero projections.
This can be seen via almost disjoint families.

\begin{lemma}\label{lem:manyproj}
There exists a family $(e_t)_{t\in T}$ of mutually unitarily equivalent nonzero projections in $\QH$ with $|T|=2^{\aleph_0}$ and
\[e_t e_s=0\quad\text{for all }t\ne s.
\]
\end{lemma}

\begin{proof}
Let $(e_n)_{n\in\mathbb N}$ be the standard basis of $\ell^2(\mathbb N)$.
Fix an almost disjoint family $\mathcal A\subseteq [\mathbb N]^{\infty}$ of infinite subsets of $\mathbb N$ with $|\mathcal A|=2^{\aleph_0}$ (existence is standard in ZFC; see e.g. \cite[Ch.~I, \S2]{FarahAQ}).
For $A\in\mathcal A$ let $P_A\in B(\ell^2(\mathbb N))$ be the diagonal projection onto $\overline{\mathrm{span}}\{e_n:n\in A\}$.
Then $P_A$ has infinite rank, hence $\pi(P_A)\neq 0$ in $\QH$.
If $A,B\in\mathcal A$ are distinct, then $A\cap B$ is finite, hence $P_A P_B$ is finite-rank and therefore compact.
Thus $\pi(P_A)\pi(P_B)=\pi(P_A P_B)=0$ in $\QH$.
Taking $T=\mathcal A$ and $e_A:=\pi(P_A)$ completes the proof.
\end{proof}

\begin{lemma}\label{lem:separated}
Let $K$ be a separable Hilbert space and let $\delta\in[0,1)$.
If $(q_i)_{i\in I}\subseteq B(K)$ are nonzero projections such that
\[\|q_i q_j\|\le \delta\qquad(i\ne j),\]
then $I$ is countable.
\end{lemma}

\begin{proof}
Pick unit vectors $\xi_i\in q_i K$.
For $i\ne j$ we have
\[|\langle \xi_i,\xi_j\rangle|=|\langle q_j\xi_i,\xi_j\rangle|\le \|q_j\xi_i\|\le \|q_j q_i\|\le \delta.
\]
Hence
$\|\xi_i-\xi_j\|^2 = 2-2\mathrm{Re}\langle \xi_i,\xi_j\rangle \ge 2(1-\delta)=:\varepsilon^2.
$
So $\{\xi_i:i\in I\}$ is an $\varepsilon$-separated subset of the unit sphere.
Since the unit sphere of a separable Hilbert space is a separable metric space, every $\varepsilon$-separated subset is countable.
\end{proof}

We also need a standard functional calculus estimate turning an ``almost projection'' into a genuine projection, which we record without proof.

\begin{lemma}\label{lem:almostproj}
Let $a\in\BH$ satisfy $a^*=a$, $\|a\|\le 1$ and $\|a^2-a\|\le \delta$ with $\delta<1/4$.
Then there exists a projection $q\in\BH$ such that $\|a-q\|\le 2\delta$.
\end{lemma}

\begin{proposition}\label{prop:calkin_kill}
One has $\varphi^\perp(x)$ is uniformly small in the operator norm for $x \in B(H)_{\le 1}$, hence
$\varphi(x)=p\varphi(x)p$ for all $x\in\BH$ after uniformly changing $\varphi$ by a  small amount.
\end{proposition}

\begin{proof}
Let $(e_t)_{t\in T}$ be the pairwise orthogonal nonzero projections in $\QH$ from Lemma~\ref{lem:manyproj}.
For each $t$, set $a_t:=\varphi^{\perp}(e_t)\in (1-p)\BH(1-p)$.
Then $a_t$ is positive and $\|a_t\|\le 1$.
Moreover, since $e_t=e_t^2$ and $\varphi^{\perp}$ is $\eta$-multiplicative,
$\|a_t^2-a_t\|=\|\varphi^{\perp}(e_t^2)-\varphi^{\perp}(e_t)^2\|\le \eta.
$
Also, for $s\ne t$, $e_s e_t=0$ gives
$\|a_s a_t\| = \|\varphi^{\perp}(e_s)\varphi^{\perp}(e_t)-\varphi^{\perp}(0)\|\le \eta.
$
Assume towards a contradiction that $a_{t_0}\ne 0$ for some $t_0$.
If $\eta<1/4$, Lemma~\ref{lem:almostproj} provides, for each $t$, a projection $q_t$ with
$\|a_t-q_t\|\le 2\eta.
$
Then for $s\ne t$,
\[
\|q_s q_t\|\le \|q_s-a_s\|\,\|q_t\| + \|a_s a_t\| + \|a_s\|\,\|q_t-a_t\|
\le 2\eta + \eta + 2\eta = 5\eta.
\]
If $a_{t_0}\ne 0$ then for $\eta$ sufficiently small the corresponding $q_{t_0}$ is nonzero, and in fact all $q_t$ are nonzero for $t \in T$, since they are pairwise equivalent and $\varphi^{\perp}$ is almost multiplicative. This contradicts  Lemma \ref{lem:separated}. Therefore $a_t=0$ for all $t\in T$, and by linearity and norm density in the $C^*$-subalgebra generated by $(e_t)$, we get that $\varphi^{\perp}(x)$ is uniformly close to zero.
\end{proof}

This finishes the proof. Indeed, Proposition~\ref{prop:calkin_kill} shows that $\varphi(\BH)\subseteq p\BH p$ after changing $\varphi$ uniformly by some small amount.
It remains to show that $\Phi$ is uniformly close to $\varphi$, but 
$$\varphi(x) \varphi(p_n) \sim_{\varepsilon} \varphi(xp_n) \sim_{\varepsilon} \Phi(xp_n) = \Phi(x) \varphi(p_n)$$ for all $n \in \mathbb N$ uniformly.

\section{Applications to rigidity properties}

\subsection{Kadison--Kastler distance and poor man's Ulam stability}\label{S.KadisonKastler}

While Proposition~\ref{P.Counterexample} implies that Ulam stability fails, in this section we prove poor man's versions  still suffices to obtain strong corona rigidity results. Following the standard notation in the study of Kadison--Kastler distance, for \cstar-subalgebras $A$ and $B$ of $B(H)$ and $\gamma>0$ we write $A\subseteq_\gamma B$ if $\dist(a,B)\leq \gamma$ for all $a\in A_1$. 
The Kadison--Kastler distance between $A$ and $B$ is defined as follows (\cite[\S 2]{CSSWW:Perturbations})
\[
d(A,B)<\gamma\Leftrightarrow
(\forall a\in A_1)(\exists b\in B) \, \|a-b\|<\gamma\text{ and }
(\forall b\in B_1)(\exists a\in A) \, \|a-b\|<\gamma. 
\]
There is a subtle difference between the definitions of $d$ and $\subseteq_\gamma$, and the former is not exactly equal to the Hausdorff distance between the unit balls of $A$ and $B$. This will make no difference for our purposes. 

The notation $A\subseteq_\gamma B$ is meaningful even when $B$ is a subset of $B(H)$ which is not necessarily a \cstar-algebra, as in the following paragraph. 

\begin{definition}
	An $\varepsilon$-$*$-homomorphism $\varphi\colon A\to B$ is called \emph{$\varepsilon$-injective} if every $a\in A_1$ satisfies $\|\varphi(a)\|\geq \|a\|-\varepsilon$. 
	It is called \emph{$\varepsilon$-surjective} if $B\subseteq _\varepsilon\varphi[A]$. A $\varepsilon$-$*$-homomorphism that is both $\varepsilon$-injective and $\varepsilon$-surjective is called $\varepsilon$-isomorphism. 
\end{definition}

\begin{proposition}\label{P.AlmostInclusion}
	Suppose that $\sfA$ is a class of \cstar-algebras Ulam stable with respect to von Neumann algebras. If $A\in \sfA$ has finite nuclear dimension $n$ and separable, and $B$ is separable, then the following are equivalent. 
	\begin{enumerate}
		\item There is an $\varepsilon$-isometric $\varepsilon$-$*$-homomorphism $\varphi\colon A\to B$ for some $\varepsilon$ such that  with $\eta=f_\sfA(\varepsilon)$  we have $\epsilon+\eta<1$ and 
		$2(\dimnuc(A)+1)(2\eta+\eta^2)(2+2\eta+\eta^2)<1/210000$. 
		\item $A$ is isomorphic to a \cstar-subalgebra of $B$. 
	\end{enumerate}
\end{proposition}

\begin{proof} Only the direct implication requires a proof. Suppose that $\varphi\colon A\to B$ is a $\varepsilon$$*$-homomorphism and fix a faithful representation of $B$ on a Hilbert space $H$. 
	Since $A\in \sfA$ there is a $*$-homomorphism $\Phi\colon A\to B(H)$ such that $\sup_{\|a\|\leq 1} \|\Phi(a)-\varphi(a)\|\leq f_\sfA(\varepsilon)$. 
	If $a\in A$ and $\|a\|=1$ then $\|\Phi(a)\|\geq \|a\|-\varepsilon-f_\sfA(\varepsilon)>0$. Since $\ker(\Phi)$ is a \cstar-algebra, this implies that $\Phi$ is injective and therefore isometric.   Write $\eta=f_\sfA(\varepsilon)$ and $\gamma=2(\dimnuc(A)+1)(2\eta+\eta^2)(2+2\eta+\eta^2)<1/210000$. 
	For $a\in A_1$ we have that $\Phi(a)\approx_{\eta} \varphi(a)\in B$, and therefore $\Phi[A]$ is an isomorphic copy of $A$ such that $\Phi[A]\subseteq_\eta B$.  Therefore \cite[Theorem~6.10]{CSSWW:Perturbations} implies that $B$ has a \cstar-subalgebra isomorphic to $A$.
\end{proof}

\begin{proposition}
	\label{P.AlmostIsomorphic}
	Suppose that $\sfA$ is a class of \cstar-algebras Ulam stable with respect to von Neumann algebras. If $A\in \sfA$ is nuclear and separable and $B$ is separable, then the following are equivalent. 
	\begin{enumerate}
		\item There is an $\varepsilon$-isometric, $\varepsilon$-surjective,  $\varepsilon$-$*$-homomorphism $\varphi\colon A\to B$ for some $\varepsilon$ such that  $\varepsilon+f_\sfA(\varepsilon)<1/420000$. 		\item $A$ is isomorphic to $B$. 
	\end{enumerate}
\end{proposition}

\begin{proof}
	Only the direct implication requires a proof.   Suppose that $\varphi\colon A\to B$ is a $\varepsilon$-$*$-homomorphism and fix a faithful representation of $B$ on a Hilbert space $H$. 
	As in the proof of Proposition~\ref{P.AlmostInclusion}, $A\in \sfA$ implies that there is an injective  $*$-homomorphism $\Phi\colon A\to B(H)$ such that (writing $\eta=\varepsilon+f_\sfA(\varepsilon)$) we have  $\Phi[A]\subseteq_{\eta} B$. 
	Moreover, for every $b\in B_1$ there is $a\in A$ such that $\varphi(a)\approx_\varepsilon b$, 
	and therefore $\Phi(a)\approx_{\eta} b$. This implies that $B\subseteq_\eta \Phi[A]$ and therefore $d(A,B)\leq \eta$. Since $\eta<1/210000$, 
	\cite[Theorem~4.3]{CSSWW:Perturbations} implies that  $A$ and $B$ are isomorphic. 
\end{proof}

See \cite[Theorem~6.10]{CSSWW:Perturbations} and  \cite[Theorem~11 (i) and (iii)]{christensen2010spatial} for related positive results. 
In \cite[Question~5.8]{vignati2022rigidity} it was asked whether there is $\varepsilon>0$ such that for unital separable \cstar-algebras, $\varepsilon$-isomorphism is equivalent to isomorphism. This question is clearly related to the well-studied question whether there is $\varepsilon>0$ such that if $A$ and $B$ are unital, separable \cstar-subalgebras of $B(H)$ with $d(A,B)<\varepsilon$ one necessarily has $A\cong B$.  The following partial answer to this question is an immediate consequence of Proposition~\ref{P.AlmostIsomorphic} and Theorem~\ref{T.Classifiable}. 

\begin{corollary}
	\label{C.e-isomorphic} There exists some $\varepsilon>0$ such that for stably finite Elliott-classifiable \cstar-algebras, $\varepsilon$-isomorphism implies isomorphism. \qed 
\end{corollary}

\subsection{Applications to corona rigidity}\label{S.CoronaRigidity}

If $\varphi_n\colon A_n\to B_n$ is an $\varepsilon_n$-$*$-homomorphism for all $n\in \bbN$, we can define $\Phi_*\colon \prod_n A_n\to \prod_n B_n$ by 
\[
\Phi_*(a_n)=(\varphi_n(b_n)). 
\]
Then the function $\Phi$ between the unit balls of the coronas $\prod_n A_n/\bigoplus_n A_n$ and $\prod_n B_n/\bigoplus_n B_n$ that makes the diagram in Fig.~\ref{Fig.Lifting} commute (the vertical arrows present quotient maps)
\begin{figure}[h]
	\begin{tikzpicture}
		\matrix[row sep=1cm,column sep=1.5cm] 
		{
			& & \node (M1) {$\prod_n A_n$}; & \node (M2) {$\prod_n B_n$};&
			\\
			& & \node (Q1) {$\prod_n A_n/\bigoplus_n A_n$}; & \node (Q2) {$\prod_n B_n/\bigoplus_n B_n$} ;
			\\
		};
		\draw (M1) edge [->] node [above] {$\Phi_*$} (M2);
		\draw (Q1) edge [->] node [above] {$\Phi$} (Q2);
		\draw (M1) edge [->] node [left] {} (Q1);
		\draw (M2) edge [->] node [left] {} (Q2);
	\end{tikzpicture}
	\label{Fig.Lifting}
	\caption{An asymptotically algebraic homomorphism.}
\end{figure}
is a $*$-homomorphism. Remarkably, by results of \cite{mckenney2020forcing} and \cite{vignati2022rigidity}, under appropriate set-theoretic assumption (forcing axioms $\OCAT$ and $\MA$), $*$-homomorphisms between many coronas are often of this form. 
 It is considerably less surprising that the Continuum Hypothesis implies the existence of isomorphisms that are not of this form and therefore leads to independence results (\cite{ghasemi2014reduced}). 
 
The assumptions $\OCAT$ and $\MA$  used in results below are forcing axioms relatively consistent with $\ZFC$, 
see \cite[\S 8.6]{FarahCT} and \cite{Ku:Set}. 
These axioms imply strong quotient rigidity in several categories, including \cstar-algebras. See \cite{farah2025corona} and \cite[\S 17]{FarahCT}. 

In the following, taken from  \cite[Definition~3.13]{mckenney2020forcing},  the corona of a nonunital \cstar-algebra~$A$ is denoted $\cQ(A)$. 

\begin{definition}\label{D.Trivial} Suppose that $A_n$ and $B_n$, for $n\in \bbN$, are unital \cstar-algebras. A function $\Phi_*\colon \prod_n A_n\to \prod_n B_n$ is called \emph{asymptotically algebraic} if 
there are a bijection $f$ between cofinite subsets of $\bbN$  and functions\footnote{Absolutely nothing is assumed on these functions; they need not be $*$-homomorphisms, linear, or continuous.} $\varphi_n\colon  A_{f(n)}\to B_n$ for $n\in \dom(f)$ such that 
\[
\Phi_*((a_n)_n)= (\varphi(a_{f(n)})_n. 
\]
An isomorphism $\Phi\colon \cQ(\bigoplus_n A_n)\to \cQ(\bigoplus_n B_n)$ is called \emph{algebraically trivial}, if there is an asymptotically algebraic lifting $\Phi_*\colon \prod_n A_n\to \prod_n B_n$ such that each of the functions $\varphi_n$ is a *-isomorphism. 
\end{definition}

Theorem~\ref{T.CoronaRigidity+} is a consequence of our results and \cite{mckenney2020forcing}. Its assumptions are for example satisfied if $A_n$ and $B_n$ are simple AF algebras (the metric approximation property is a weakening of the completely positive approximation property, which is equivalent to nuclearity). 
The following is  Theorem~\ref{T.CoronaRigidity}. 

\begin{theorem} \label{T.CoronaRigidity+} Assume $\OCAT$ and $\MA$. 
	If $A_n$, $B_n$, for $n\in \bbN$, are stably finite,  Elliott-classifiable unital \cstar-algebras,  then the following are equivalent. 
	\begin{enumerate}
		\item The coronas of $\bigoplus_n A_n$ and $\bigoplus_n B_n$ are isomorphic. 
		\item There are a bijection $\pi$ between cofinite subsets of $\bbN$ such that $A_{\pi(n)}\cong B_n$ for all $n\in \dom(\pi)$. 
	\end{enumerate}
\end{theorem}

The proof uses the following theorem, proved in \cite{mckenney2020forcing} but not stated in a form convenient for us; we will explain why it is true for reader's convenience. 

\begin{theorem}\label{T.McKV}
Assume $\OCAT$ and $\MA$.  If \cstar-algebras $A_n$, $B_n$ are unital, separable, have no central projections, that $\bigoplus_n A_n$ has the metric approximation property   then, with $A=\bigoplus_n A_n$ and $B=\bigoplus_n B_n$, every isomorphism $\Phi\colon \cQ(A)\to \cQ(B)$ is algebraically trivial. 
\end{theorem}

\begin{proof}  \cite[Theorem~C]{mckenney2020forcing} states that our assumptions imply that $\Phi$ is topologically trivial, which means that the set ($\pi_A$ and $\pi_B$ are the quotient maps from the multiplier algebras to the corresponding coronas and $\cM(A)_1$ denotes the unit ball of the multiplier algebra)
\[
\{(a,b)\in \cM(A)_1\times\cM(B)_1: \Phi(\pi_A(a))=\pi_B(b)\}
\]
is Borel in the strict topology (\cite[Definition~A]{mckenney2020forcing}).   \cite[Theorem~E]{mckenney2020forcing} asserts that, under our assumptions, every topologically trivial isomorphism has an asymptotically algebraic lifting. 
\end{proof}

\begin{proof}[Proof of Theorem~\ref{T.CoronaRigidity+}]
	Assume $\OCAT$ and $\MA$. 
	If $A_n$, $B_n$, for $n\in \bbN$, are stably finite,  Elliott-classifiable unital \cstar-algebras and that $\Phi$ is an isomorphism between the coronas of $\bigoplus_n A_n$ and $\bigoplus_n B_n$. 
Since these algebras are simple they have no central projections and 	
 $\OCAT$, $\MA$, and Theorem~\ref{T.McKV} imply that $\Phi$ has an asymptotically algebraic lifting. Fix $f$ and $\varphi_n$, for $n\in \bbN$, as in Definition~\ref{D.Trivial}. 

A simple argument (\cite[Lemma 3.12]{mckenney2020forcing}) shows that there are $\varepsilon_n>0$, for $n\in \bbN$,   such  that $\lim_n \varepsilon_n=0$ and $\varphi_n\colon A_{f(n)}\to B_n$ is an $\varepsilon_n$-$*$-isomorphism. See  \cite[Theorem~3.16]{mckenney2020forcing}.  By Corollary~\ref{C.e-isomorphic}, there is an isomorphism $\Psi_n\colon A_{f(n)}\to B_n$ for all sufficiently large $n$. Let $\Psi_n(a)=0$ for other values of $n$.  Then the function $(a_n)_n\mapsto (\Psi_n(a_{f(n)}))_n $ is a $*$-homomorphism from $\prod_n A_n$ to $\prod_n B_n$ that is a lifting of a trivial isomorphism between $\cQ(A)$ and $\cQ(B)$. 
\end{proof}

Notably, the isomorphism between coronas constructed in the proof of Theorem~\ref{T.CoronaRigidity+} may differ from the original isomorphism. In light of Corollary~\ref{P.Counterexample}, with $A$ denoting the unitization of $C([0,1],\cK(H))$, there is an automorphism of $A_\infty$ that has no 
lifting of the form $(a_n)\mapsto \Psi_n(a_{f(n)})$ for an isomorphism $\Psi_n\colon A\to A$.

\begin{theorem}\label{T.CoronaEmbedding+} Assume $\OCAT$ and $\MA$. 
	Suppose that $A$ and $B$ are separable, unital, \cstar-algebras and that $A$ is stably finite and Elliott classifiable. Then the following are equivalent. 
	\begin{enumerate}
		\item $A_\infty$ embeds into $\cQ(B\otimes \cK)$ (not necessarily unitally). 
		\item $A$ embeds into $B\otimes \cK$. 
	\end{enumerate}
\end{theorem}

\begin{proof} 
Only the direct implication requires a proof. Modulo Proposition~\ref{P.AlmostInclusion}, this proof is analogous to the proof of \cite[Theorem~D]{vignati2022rigidity}. 
Suppose that there is an embedding of $A_\infty$ into $\cQ(B\otimes \cK)$. 
By \cite[Theorem~4.3]{vignati2022rigidity}, there are a sequence $\varepsilon_n\to 0$, $k(n)\geq 1$, $m(n)\geq 1$,  and $\varepsilon_n$-injective $\varepsilon_n$-$*$-homomorphism\footnote{After unravelling the notation, this is exactly what the first displayed formula in Theorem~4.3 asserts.}
\[
\varphi_n\colon A^{k(n)}\to M_{m(n)}(B). 
\]
We may assume $k(n)=1$, and then Proposition~\ref{P.AlmostInclusion} implies that $A$ is isomorphic to a \cstar-subalgebra of $M_{m(n)}(B)$, which implies the desired conclusion. 
\end{proof}

We can now prove Corollary~\ref{C.Independence}. 

\begin{corollary}\label{C.Independence.+}
	\begin{enumerate}
		\item The assertion that $(M_{2^\infty})_\infty$ embeds into $\cQ(M_{3^\infty}\otimes \cK(H))$ is independent from ZFC. 
		\item 	There are UHF \cstar-algebras $A_n$, $B_n$, for $n\in \bbN$ such that the assertion that $\prod_n A_n/\bigoplus_n A_n$ and $\prod_n B_n/\bigoplus_n B_n$ are isomorphic is independent from ZFC. 
	\end{enumerate}
\end{corollary}

\begin{proof}
	For the first part, by  Theorem~\ref{T.CoronaEmbedding+} $\OCAT$ and $\MA$ imply that $(M_{2^\infty})_\infty$ does not embed into $\cQ(M_{3^\infty}\otimes \cK(H))$. 
	
We will prove that the Continuum Hypothesis (CH) implies the opposite conclusion. 
	By \cite[Theorem~A]{farah2017calkin}, 
	every \cstar-algebra of density character $\aleph_1$ embeds into the Calkin algebra $\cQ(H)$ and therefore CH implies that $(M_{2^\infty})_\infty$ embeds into $\cQ(H)$. Since $\cQ(H)$ embeds into $\cQ(B\otimes \cK(H))$ for every unital $B$, the conclusion follows.

	The second part makes some use of model theory. As in the proof of \cite[Lemma~5.2]{ghasemi2014reduced}, there is an increasing sequence of prime numbers $p_j$, for $j\in \bbN$, such that for every sentence $\varphi$ in the language of \cstar-algebras $\lim_{j\to \infty} \varphi^{M_{p_j^\infty}}$ exists. By Ghasemi's Feferman--Vaught theorem, as explained in the proof of \cite[Proposition~5.3]{ghasemi2014reduced}, this implies that for every two increasing sequences $j(i)$, $j'(i)$, for $i\in \bbN$, the coronas of $\bigoplus_i  {M_{p_j(i)^\infty}}$ and 
$\bigoplus_i  {M_{p_j'(i)^\infty}}$ are elementarily equivalent, and therefore isomorphic if CH holds. 
On the other hand, the coronas of
$\bigoplus_i  {M_{p_{2j}^\infty}}$
and $\bigoplus_i  {M_{p_{2j+1}^\infty}}$ are not isomorphic by Theorem~\ref{T.CoronaRigidity+}. 
\end{proof}

\section{Open problems}

We end with a list of open problems that seem natural in light of our results.

\begin{question}
\label{P.UlamStability} Is the class of all stably finite, nuclear, separable \cstar-algebras Ulam stable with respect to von Neumann algebras?
\end{question}

\begin{question} Is the class of Ulam stable \cstar-algebras with respect to von Neumann algebras closed under passing to full hereditary subalgebras?
\end{question}

\begin{question}
What about Ulam stability with respect to von Neumann algebras for purely infinite nuclear \cstar-algebras?	
\end{question}

\begin{question} Is there a \cstar-algebra that is not Ulam stable with respect to von Neumann algebras?
\end{question}

\section*{Acknowledgments} IF would like to thank Ryszard Nest and Stuart White for helpful remarks  and to the Department of Mathematics at TU Dresden for their hospitality in January 2026 when this work commenced.  We would also like to thank Austin Shiner and Stuart White for remarks on the first draft of this paper.

\bibliography{abelian}

\end{document}